\documentclass[a4,11pt]{article}
\usepackage{amsthm}

\usepackage[a4paper]{geometry}
\newcommand{\newnumbered}[2]{\newtheorem{#1}[theorem]{#2}}
\newcommand{\newunnumbered}[2]{\newtheorem{#1}[theorem]{#2}}
\newcommand{\Title}[2]{\title{#1}\newcommand{\Acknowledgements}{\section*{Acknowledgements} #2}}
\newcommand{\Author}[2][]{\author{#2}}
\newcommand{\Comma}{\and}

\newcommand{\br}{, }
\newcommand{\fs}{. }
\newcommand{\thanksone}[3][]{#1\thanks{#3\email{\tt #2}}}
\newcommand{\thankstwo}[3][]{#1\thanks{#3\email{\tt #2}}}

\newcommand{\email}[1]{#1}
\newcommand{\classno}[2][2000]{}
\newcommand{\printscl}{}
\newcommand{\mktitle}{\maketitle}
\newcommand{\mkabstitle}{}

\usepackage[utf8]{inputenc}
\usepackage{amsfonts}
\usepackage{amsmath}
\usepackage{amssymb}
\usepackage[mathscr]{eucal}
\usepackage{color}
\usepackage{tikz}
\usepackage[colorlinks=true,citecolor=black,linkcolor=black,urlcolor=blue]{hyperref}
\usepackage{comment}
\newtheorem{theorem}{Theorem}[section] 
\newtheorem{lemma}[theorem]{Lemma}     
\newtheorem{claim}[theorem]{Claim}     
\newtheorem{corollary}[theorem]{Corollary}

\newnumbered{assertion}{Assertion}    
\newnumbered{conjecture}{Conjecture}  
\newnumbered{definition}{Definition}
\newnumbered{hypothesis}{Hypothesis}
\newnumbered{remark}{Remark}
\newnumbered{note}{Note}
\newnumbered{observation}{Observation}
\newnumbered{problem}{Problem}
\newnumbered{question}{Question}
\newnumbered{algorithm}{Algorithm}
\newnumbered{example}{Example}
\newunnumbered{notation}{Notation} 
\numberwithin{equation}{section}

\newcommand{\DOI}[1]{\href{http://dx.doi.org/#1}{\texttt{doi:#1}}}

\newcommand{\noop}[1]{}

\newcommand\rank{\mathrm{rank}}

\newcommand\opk[1]{\mathop{\mathsf{#1}}\nolimits}
\newcommand\Harm{\opk{Harm}}
\newcommand\Hom{\opk{Hom}}

\newcommand{\dd}{\; \mathrm{d}}

\newcommand\ip[2]{\langle#1,#2\rangle}

\newcommand\Om{\Omega}

\allowdisplaybreaks

\pgfdeclarelayer{edgelayer}
\pgfdeclarelayer{nodelayer}
\pgfsetlayers{edgelayer,nodelayer,main}

\usetikzlibrary{decorations.markings}

\tikzstyle{new}=[circle,  minimum width=4pt,inner sep=0pt, fill=black,draw=black]
\tikzstyle{none}=[circle,fill=white,draw=black]
\tikzstyle{n}=[shape=rectangle,minimum width=1pt,inner sep=0pt, fill=none,draw=none]
\tikzstyle{emph}=[circle,  minimum width=4pt,inner sep=0pt, fill=magenta,draw=magenta]

\tikzset{directed/.style={decoration={
  markings,
  mark=at position .6 with {\arrow{>}}},postaction={decorate}}}

\begin{document}
\Title{On the multiplicities of digraph eigenvalues}{%
Alexander Gavrilyuk is supported by
Basic Science Research Program through the National Research Foundation of Korea (NRF) funded
by the Ministry of Education (grant number NRF-2018R1D1A1B07047427).
Sho Suda is supported by JSPS KAKENHI Grant Number 18K03395.
}

\Author[Alexander L. Gavrilyuk, Sho Suda]{%
\thanksone[Alexander L. Gavrilyuk]{alexander.gavriliouk@gmail.com}{%
Pusan National University\br
2, Busandaehak-ro 63beon-gil\br
Geumjeong-gu, Busan, 46241\br
Republic of Korea\fs
}
\Comma
\thankstwo[Sho Suda]{ssuda@nda.ac.jp}{%
Department of Mathematics\br
National Defense Academy of Japan\br
2-10-20 Hashirimizu, Yokosuka\br
Kanagawa, 239-8686\br
Japan\fs
}
}

\classno{05E30 (primary), 05B15 (secondary)}

\date{\today}

\mktitle

\begin{abstract}
We show various upper bounds for the order of a digraph (or a mixed graph) whose Hermitian adjacency matrix
has an eigenspace of prescribed codimension. In particular, this generalizes the so-called absolute bound
for (simple) graphs first shown by Delsarte, Goethals, and Seidel (1977) and extended by
Bell and Rowlinson (2003).
In doing so, we also adapt the Blokhuis' theory (1983) of harmonic analysis in 
real hyperbolic spaces to that in complex hyperbolic spaces. 
\printscl
\end{abstract}

\mkabstitle

\section{Introduction}
Although it is likely that almost all graphs have simple eigenvalues only \cite{Tao},
many interesting highly structured graphs (such as strongly regular and distance-regular graphs)
have eigenvalues of large multiplicities relative to the number of vertices.
It is then natural to ask whether the multiplicity of an eigenvalue can be bounded above.
In 1977, Delsarte, Goethals, and Seidel obtained the following result regarding this question
as a by-product of the theory of spherical codes and designs.

\begin{theorem}[\cite{DGS}]\label{theo-DGS}
A regular (simple) graph on $n$ vertices, whose $(0,1)$-adjacency matrix
has the smallest eigenvalue $<-1$ of multiplicity $n-d$, satisfies
\[
n\leq \frac{1}{2}d(d+1)-1=\frac{1}{2}(d-1)(d+2).
\]
\end{theorem}

This bound is sharp.
Namely, recall that a {\bf strongly regular} graph $\Gamma$
can be defined as a connected $k$-regular graph with precisely three distinct eigenvalues.
Then $k$ is the eigenvalue of multiplicity 1, and if $m_1\geq m_2$
denote the multiplicities of the two other eigenvalues,
then the order $n$ of $\Gamma$ satisfies
\begin{equation}\label{eq-abs-bound-SRG}
n\leq \frac{1}{2}m_2(m_2+3),
\end{equation}
which is called the {\bf absolute bound} for strongly regular graphs \cite{Neu,Sei}.
It is easily seen that the absolute bound follows
from Theorem \ref{theo-DGS}.
A strongly regular graph that attains the absolute bound is said
to be {\bf extremal}. In fact, an extremal strongly regular graph
is a pentagon, a complete multipartite graph or a so-called Smith graph.
The only known examples of extremal Smith graphs are
the Schl\"{a}fli graph, the McLaughlin graph and their complements.

Bell and Rowlinson \cite{BR} (see also \cite{RP}) generalized Theorem \ref{theo-DGS} as follows.


\begin{theorem}\label{theo-BR-2}
A $k$-regular (simple) graph on $n$ vertices,
whose $(0,1)$-adjacency matrix has eigenvalue
$\lambda\notin\{0,-1,k\}$ of multiplicity $n-d$, $d\geq 2$,
satisfies\footnote{Note that  \cite[Theorem~3.1]{BR} requires $d>2$, but in fact the result also holds for $d=2$.}
\begin{equation}\label{eq-DGS}
 n\leq \frac{1}{2}d(d+1)-1,
\end{equation}
and equality holds if and only if the graph is
an extremal strongly regular graph, with $\lambda$
as its eigenvalue of greatest multiplicity.
\end{theorem}

For a large class of strongly regular graphs
Neumaier \cite{Neu} showed that 
\begin{equation}\label{eq-abs-bound-Neu}
n\leq \frac{1}{2}m_i(m_i+1),
\end{equation}
which improves upon the bound in Eq.\ \eqref{eq-abs-bound-SRG}; in Section \ref{sect:final} we discuss a generalization 
of this result for regular connected graphs.

The goal of the present paper is to prove analogues
of the above mentioned results for digraphs.
A {\bf digraph} (a {\bf directed} or {\bf mixed} graph) $\Delta$ consists
of a finite set $V$ of vertices together with a subset $E\subseteq  V\times V$
of ordered pairs of elements of $V$ called {\bf arcs} or {\bf directed edges}.
If $(x,y)\in E$, then we write $x\rightarrow y$. If both $x\rightarrow y$ and
$y\rightarrow x$, then the pair $\{x,y\}$ forms a {\bf digon} of $\Delta$,
which may be thought of as an undirected edge, and we write $x\sim y$ in this case.

Let $\Delta=(V,E)$ be a digraph on $n$ vertices.
For a complex number $\omega\notin \mathbb{R}$ with absolute value $1$, 
we consider the {\bf Hermitian adjacency matrix} $H=H_{\omega}(\Delta)\in \mathbb{C}^{V\times V}$ of $\Delta$
with entries given by:
\begin{align*}
(H)_{xy}=\begin{cases}
1 & \text{ if }x\sim y,\\
\omega & \text{ if }x\rightarrow y,\\
\overline{\omega} & \text{ if }y\leftarrow x,\\
0 & \text{ otherwise},
\end{cases}
\end{align*}
where $\overline{\omega}$ denotes\footnote{In what follows, 
$\overline{\omega}$ denotes 
the complex conjugate of a complex scalar, 
while $M^*$ stands for the Hermitian transpose of a matrix $M$.} the conjugate
of $\omega$.

The notion of the Hermitian adjacency matrix (with $\omega={\bf i}$)
was introduced by Liu and Li \cite{LL} and independently by Guo and Mohar \cite{GM}.
Mohar \cite{M} introduced another type of Hermitian adjacency matrices,
which coincides with ours (see also \cite{Kubota}) in the important case $\omega=\frac{1+{\bf i}\sqrt{3}}{2}$, which 
is the primitive sixth root of unity.

Since $H$ is a Hermitian matrix, the eigenvalues of $H$ are all real,
and their algebraic and geometric multiplicities coincide. 
The eigenvalues interlacing property, which is useful in the study
of simple graphs \cite{H}, also holds for eigenvalues of $H$ and
those of its principal submatrices. On the other hand, as Guo and Mohar
notice \cite{GM}, the eigenvalues of $H$ may behave differently (with respect to
the digraph structure) compared with the eigenvalues of $(0,1)$-adjacency
matrices of simple graphs. For example, some digraph invariants such as
diameter, minimum degree, and number of connected components cannot
be bounded by the spectrum of $H$.

Recall that an eigenvalue $\lambda$ is said to be {\bf non-main}
if its eigenspace is contained in $\langle{\bf 1}\rangle^{\perp}$,
where ${\bf 1}$ is the all-one vector; otherwise we call $\lambda$
the {\bf main} eigenvalue. (In particular, if a simple graph is regular with
valency $k$, then all its eigenvalues but $k$ are non-main.)
With this notation, our first main result reads as follows.


\begin{theorem}\label{theo-main}
Let $H=H_{\omega}(\Delta)$ be a Hermitian adjacency matrix of a digraph of order $n$ with $s$
distinct eigenvalues $\lambda_1>\lambda_2>\cdots >\lambda_s$ of respective multiplicities $m_1,m_2,\ldots,m_s$.
Set $d_i=n-m_i$.
  Then, for an eigenvalue $\lambda_i$ with $d_i\geq 2$ and $\lambda_i\notin\{0,-1\}$, the following holds.
  \begin{itemize}
  \item[$(i)$] 
  If $i\in \{1,2,\ldots,s\}$ and $\lambda_i$ is non-main, then
\begin{align*}
n\leq \begin{cases}
\frac{3d_i(d_i+1)}{2}-1 & \text{~if~}\omega\ne \frac{-1\pm {\bf i}\sqrt{3}}{2},\\
\frac{d_i(d_i+5)}{2} & \text{~if~}\omega=\frac{-1\pm {\bf i}\sqrt{3}}{2}.
\end{cases}
\end{align*}
  \item[$(ii)$] If $i\in \{1,s\}$, then
\begin{align*}
n\leq \begin{cases}
\frac{d_i(3d_i+5)}{2}-1 & \text{~if~}\lambda_i\text{~is main},\\
\frac{(d_i-1)(3d_i+2)}{2}-1 & \text{~if~}\lambda_i\text{~is non-main}.
\end{cases}
\end{align*}
  \end{itemize}
\end{theorem}

To put this result in a more general context, we observe that
the codimension $d$ of the $\lambda$-eigenspace equals
the rank of $H-\lambda I$. Let $N(d,L)$ denote the maximum size of a square
matrix of rank at most $d$, whose off-diagonal entries belong to a set $L$.
Many important applications of linear algebra to combinatorics reduce to bounding
the size of a matrix with few distinct entries and a given rank, i.e., $N(d,L)$.
Recently Bukh \cite{Bukh} obtained some general asymptotic results regarding this problem;
in particular, he showed that
\begin{equation}\label{eq-Bukh}
  N(d,L)\leq {d+|L|\choose |L|}\sim \frac{d^{|L|}}{|L|!}+O(d^{|L|-1}).
\end{equation}

It is observed in \cite{Bukh} that using the application-specific structure of a matrix
may improve upon the upper bound \eqref{eq-Bukh}.
For example, if $\lambda_{\rm min}\ne 0$ is the least eigenvalue
of multiplicity $m$ of the $(0,1)$-adjacency matrix $A$, then $\frac{-1}{\lambda_{\rm min}}A+I$
is positive semidefinite of rank $d=n-m$.
Thus, it can be seen as the Gram matrix of a set of $n$ unit vectors in $\mathbb{R}^{d}$
with two distinct inner products, i.e., a spherical 2-distance set. 
In order to bound $n$, Delsarte, Goethals and Seidel \cite{DGS} used
some elements of the theory of harmonic analysis on spheres.
(A bit worse bound can be shown by a much simpler argument of Koornwinder \cite{K}.)
Nevertheless, the bound from Theorems \ref{theo-DGS} and \ref{theo-BR-2}
is just slightly better than the one in \eqref{eq-Bukh}.

The same phenomenon happens when $A$ is a symmetric $(0,\pm 1)$-matrix,
i.e., the adjacency matrix of a signed graph, see a recent result in \cite{Rsign}.
It is then remarkable that the bounds in Theorem \ref{theo-main} significantly
improve on the one in \eqref{eq-Bukh}. 

The proofs of the above-mentioned results are variations of the polynomial method.
In particular, Bell and Rowlinson \cite{BR} used the polynomial method combined
with the star complement technique, which allows to analyze
the structure of a graph with equality in Theorem \ref{theo-BR-2}.

In Section \ref{ssect:star}, we observe that the theory of star complements in
simple graphs can be extended  to that of the Hermitian adjacency matrices of digraphs. 
This provides us a tool to prove Theorem \ref{theo-main}(i), see Section \ref{ssect:starbound}.
The proof of Theorem \ref{theo-main}(ii) in Section \ref{sect:codes} is based 
on the harmonic analysis of codes of the complex unit sphere,
which was developed by Roy and Suda \cite{RS}, see Section \ref{sect:harm}.
(It is interesting that these approaches give different bounds 
unlike the case of $(0,1)$-adjacency matrices of simple graphs.)

Further, Blokhuis \cite{Bl} extended the idea of harmonic analysis on spheres to an arbitrary (not necessarily positive-definite) inner product space,
which can be used to generalize Theorem \ref{theo-DGS} to any eigenvalue of $A$.
In Section \ref{sect:Cpq}, we adapt the Blokhuis' theory to a complex hyperbolic space, 
which is then used to obtain upper bounds on the order of a digraph with Hermitian adjacency matrix $H$
in terms of inertia of $H-\lambda I$ for an eigenvalue $\lambda$.

\begin{theorem}\label{theo-main-2}
Let $H=H_\omega(\Delta)$ be a Hermitian adjacency matrix of a digraph of order $n$ with $s$
distinct eigenvalues $\lambda_1>\lambda_2>\cdots >\lambda_s$ of respective multiplicities $m_1,m_2,\ldots,m_s$.
Set $p_i=\sum_{j=i+1}^{s}m_j,q_i=\sum_{j=1}^{i-1}m_j$ if $\lambda_i>0$, or $p_i=\sum_{j=1}^{i-1}m_j,q_i=\sum_{j=i+1}^{s}m_j$ if $\lambda_i<0$.
Then, for $i\in\{2,\ldots,s-1\}$, the number $2n$ is bounded above by the corresponding value in Table \ref{tab:Cpq}.
\begin{table}[]
\begin{tabular}{|c|c|c|c|}
\hline
           $Re(\omega)$               & $\lambda_i<-1$ & $-1<\lambda_i<0$ & $0<\lambda_i$ \\ \hline
$<-\frac{1}{2}$ & $p_i^2+4 p_i q_i+q_i^2+p_i +5q_i$ & $2p_i^2+2q_i^2+2p_iq_i+4p_i$  & $p_i^2+4 p_i q_i+q_i^2+5 p_i+q_i$\\ \hline
$=-\frac{1}{2}$ & $p_i^2+q_i^2+p_i +5q_i$  & $2p_iq_i+4p_i$ & $p_i^2+q_i^2+5p_i+q_i$ \\ \hline
$>-\frac{1}{2}$ & $3p_i^2+3q_i^2+p_i+5q_i$ & $6p_i q_i+4 p_i$ & $3p_i^2+3q_i^2+5p_i+q_i$ \\ \hline
\end{tabular}
\caption{\label{tab:Cpq}Upper bounds for Theorem \ref{theo-main-2}.}
\end{table}
\end{theorem}

Consider the following digraph:
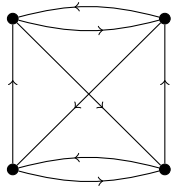
\begin{figure}[ht!]
\centering
\begin{tikzpicture}
	\begin{pgfonlayer}{nodelayer}
\node [style=new] (3) at (-1, 2.5) {};
		\node [style=new] (4) at (-1, 0.5) {};
		\node [style=new] (5) at (1, 2.5) {};
		\node [style=new] (6) at (1, 0.5) {};
	\end{pgfonlayer}
	\begin{pgfonlayer}{edgelayer}
\draw [directed, bend right=15] (3) to (5);
		\draw [directed, bend right=15] (4) to (6);
		\draw [directed] (4) to (3);
		\draw [directed] (3) to (6);
		\draw [directed] (6) to (5);
		\draw [directed] (5) to (4);
		\draw [directed, bend right=15] (6) to (4);
		\draw [directed, bend right=15] (5) to (3);
	\end{pgfonlayer}
\end{tikzpicture}
\end{figure}\\
The eigenvalues of its Hermitian adjacency matrix $H_{\bf i}$ are $-3$ of multiplicity $1$ and $1$ of multiplicity $3$ (see \cite[Proposition 5.2]{GM}). Thus, the (main) eigenvalue $1$ has codimension $d=1$ and it attains equality in $n\leq \frac{d(3d+5)}{2}$ (see Section \ref{sect:codes}), 
which is slightly improved in Theorem \ref{theo-main}(ii) for the case $d>1$.
Besides this example, we don't know much about tightness of the bounds obtained. This and some other open questions are discussed in Section \ref{sect:final}.

\section{Star complements}\label{sect:star}

In this section, we extend the theory of star complements in simple graphs
to that in digraphs, and then use it to prove Theorem \ref{theo-main}(i).

\subsection{Star complements in Hermitian matrices}\label{ssect:star}

Let $\Gamma=(V,E)$ be a simple graph. The eigenvalues of $\Gamma$ are
the eigenvalues of its $(0,1)$-adjacency matrix. Let $\Gamma$
have an eigenvalue $\lambda$ of multiplicity $m$.
A {\bf star set} for $\lambda$ in $\Gamma$ is a subset $X$ of $V$
such that $|X|=m$ and $\lambda$ is not an eigenvalue of the graph
induced on $V\setminus X$. If this is the case, the induced subgraph
on $V\setminus X$ is called a {\bf star complement} for $\lambda$
in $\Gamma$. Star sets and star complements exist for any eigenvalue
in any graph. We refer the reader to \cite{RT} for more results on star
complements and their applications.

The notion of star sets and star complements in simple graphs can be
straightforwardly extended to those in the more general setting of
Hermitian matrices, in particular, Hermitian adjacency matrices of digraphs.
We only show the following two results, which are essential
for proving Theorem \ref{theo-main}, while the comprehensive theory
of star sets and star complements in digraphs can be elaborated elsewhere.

Let $V$ be a finite set with $n$ elements, and
$H\in \mathbb{C}^{V\times V}$ be a Hermitian matrix with an eigenvalue $\lambda$ of multiplicity $m$.
We define a {\bf star set} for $\lambda$ in $H$ to be
a subset $X\subseteq V$ such that $|X|=m$ and $\lambda$ is not
an eigenvalue of the principal submatrix of $H$ corresponding
to $V\setminus X$ (i.e., with the rows and columns in $V\setminus X$).
If $X$ is a star set for $\lambda$ in $H$, then
we call $V\setminus X$ a {\bf star complement}
for $\lambda$ in $H$. When $H$ is the Hermitian adjacency matrix
of a digraph $\Delta$, by eigenvalues, star sets and star complements
of $\Delta$ we mean those of $H$, respectively.

\begin{lemma}\label{lemma-starexistence}
Let $H$ have an eigenvalue $\lambda$ of multiplicity $m$.
Then there exists a set $X\subseteq V$ such that $|X|=m$
and $\lambda$ is not an eigenvalue of the principal submatrix
of $H$ corresponding to $V\setminus X$, i.e., $X$ is a star set
for $\lambda$ in $H$, and $V\setminus X$ is a star complement
for $\lambda$ in $H$.
\end{lemma}
\begin{proof}
   Since $\lambda I-H$ has rank $n-m$, it has a principal submatrix $\lambda I-C$
   of order and rank $n-m$; hence $C$ has no eigenvalue $\lambda$.
\end{proof}

\begin{theorem}[{cf. \cite[Theorems~7.4.1,~7.4.4]{CRS}}]\label{theo-starset}
Let $X$ be an $m$-element subset of $V$,
and write $H$ partitioned as follows:
\begin{equation*}
\left(\begin{matrix}
  H_X & B^* \\
  B & C
\end{matrix}\right),
\end{equation*}
where $H_X$ is the principal submatrix of $H$
corresponding to $X$. Then $X$ is a star set
for an eigenvalue $\lambda$ in $H$ if and only if
$\lambda$ is not an eigenvalue of $C$ and
\begin{equation}\label{eq-reconstruction}
  \lambda I-H_X = B^*(\lambda I-C)^{-1}B,
\end{equation}
in which case the $\lambda$-eigenspace of $H$
consists of the vectors
\begin{equation}
  \left(
  \begin{matrix}
    {\bf x} \\
    (\lambda I-C)^{-1}B{\bf x}
  \end{matrix}
  \right)\quad ({\bf x}\in \mathbb{C}^m).
\end{equation}
\end{theorem}
\begin{proof}
  Let $X$ be a star set for $\lambda$, and so 
  $\lambda$ is not an eigenvalue of $C$.
  It follows that
  \begin{equation*}
  \lambda I-H =
    \left(\begin{matrix}
      \lambda I - H_X & -B^* \\
      -B & \lambda I - C
    \end{matrix}\right),
  \end{equation*}
  where both $\lambda I-H$ and $\lambda I-C$ are of rank $n-m$.
  Therefore, there exists a matrix $L$ such that
  \begin{equation*}
    \left(\begin{matrix}
      \lambda I - H_X & -B^*
    \end{matrix}\right) = L
    \left(\begin{matrix}
      -B & \lambda I - C
    \end{matrix}\right),
  \end{equation*}
  whence
  $\lambda I-H_X=-LB$ and $-B^*=L(\lambda I-C)$, and Eq.\ \eqref{eq-reconstruction} follows.
  The rest (including the ``only if'' part) of the proof follows from direct calculations.
\end{proof}


In what follows, let $V=\{1,2,\ldots,n\}$, $\lambda$ be an eigenvalue of $H$
of multiplicity $m$ and write $d=n-m$. By Lemma \ref{lemma-starexistence}, there exists
a star set $X$ for $\lambda$ in $H$. Without loss of generality, we assume that $X=\{1,2,\ldots,m\}$.
Following \cite{BR}, we extend the notation of Theorem \ref{theo-starset} as follows:
\begin{itemize}
  \item define
\begin{equation}\label{eq-inprod}
  \langle {\bf x}, {\bf y}\rangle := {\bf x}^*(\lambda I-C)^{-1}{\bf y}~~~({\bf x},{\bf y}\in \mathbb{C}^d),
\end{equation}
  \item denote the columns of $B$ by ${\bf b}_u$, where $u\in \{1,2,\ldots,m\}$,
  \item and define $S:=(B\mid C-\lambda I)$ with columns ${\bf s}_u$, $u\in \{1,2,\ldots,n\}$,
\end{itemize}
so that, in particular,
${\bf s}_u={\bf b}_u$ for $u\in \{1,2,\ldots,m\}$. Then
\begin{equation}\label{eq-H}
  \lambda I-H = S^*(\lambda I-C)^{-1}S.
\end{equation}

Let $\mathcal{E}(\lambda)$ be the $\lambda$-eigenspace of $H$, and $\mathcal{E}(\lambda)^{\perp}$
its orthogonal complement in $\mathbb{C}^n$.

\begin{lemma}\label{lemma-split}
  Let ${\bf w}\in \mathcal{E}(\lambda)^{\perp}$, and write ${\bf w}=({\bf p}\mid {\bf q})^{\top}=(w_1,\ldots,w_n)^{\top}$, where
  ${\bf p}=(w_1,\ldots,w_m)^{\top}$, ${\bf q}=(w_{m+1},\ldots,w_n)^{\top}$. Then,
  for $u\in \{1,2,\ldots,n\}$, we have
  \begin{equation*}
    \langle {\bf s}_u,{\bf q}\rangle = -w_u.
  \end{equation*}
\end{lemma}
\begin{proof}
Let $\{{\bf e}_1,\ldots,{\bf e}_m\}$ be the standard basis of $\mathbb{C}^m$, and let
$\{{\bf e}_{m+1},\ldots,{\bf e}_n\}$ be the standard basis of $\mathbb{C}^d$.
Since ${\bf w}\in \mathcal{E}(\lambda)^{\perp}$, Theorem \ref{theo-starset} implies that,
for $u\in \{1,2,\ldots,m\}$,
\begin{equation*}
  ({\bf p}^* \mid {\bf q}^*)\left(
  \begin{matrix}
    {\bf e}_u \\
    (\lambda I-C)^{-1}{\bf b}_u
  \end{matrix}
  \right)
  = 0,
\end{equation*}
whence $\langle {\bf q},{\bf s}_u\rangle = -\overline{w_u}$ and so $\langle {\bf s}_u,{\bf q}\rangle=-w_u$.

For $u\in \{m+1,\ldots,n\}$, we have
\begin{equation*}
  \langle {\bf s}_u,{\bf q}\rangle = {\bf e}_u^{\top}(C-\lambda I)(\lambda I-C)^{-1}{\bf q}=-{\bf e}_u^{\top}{\bf q}=-w_u,
\end{equation*}
which completes the proof.
\end{proof}

\subsection{The proof of Theorem \ref{theo-main}(i)}\label{ssect:starbound}

Let $\Delta$ be a digraph with Hermitian adjacency matrix $H=H_{\omega}(\Delta)$.
Following the notation from Section \ref{ssect:star},
assume that $\Delta$ satisfies the condition of Theorem \ref{theo-main}
and $\lambda$ is one of its non-main eigenvalues.                                              
By Eqs. \eqref{eq-inprod}, \eqref{eq-H}, for all vertices $u,v$ of $\Delta$, we have:
\begin{equation}
  \langle {\bf s}_u,{\bf s}_v\rangle = \left\{\begin{matrix}
                                               \lambda & \text{if~}u=v, \\
                                               -1 & \text{if~}u\sim v, \\
                                               -\omega & \text{if~}u\rightarrow v,\\
                                               -\overline{\omega}  & \text{if~}u\leftarrow v, \\
                                               0  & \text{~otherwise,}
                                             \end{matrix}\right.
\end{equation}
and note that $\langle {\bf s}_u,{\bf s}_v\rangle=\overline{\langle {\bf s}_v,{\bf s}_u\rangle}$.

Theorem \ref{theo-main}(i) follows from Theorems \ref{th-case1-improv} and \ref{th-case2}
proven in the following two sections.

\subsubsection{Case $\omega\ne \frac{-1\pm{\bf i}\sqrt{3}}{2}$}

Assume that $\omega\ne \frac{-1\pm{\bf i}\sqrt{3}}{2}$.
Then $\omega^2+\overline{\omega}^2-\omega-\overline{\omega}\neq 0$.
Define functions $F_1,\ldots,F_n$ by 
\begin{equation*}
  F_u({\bf x})=a\langle {\bf s}_u,{\bf x}\rangle^2+\langle {\bf s}_u,{\bf x}\rangle\langle {\bf s}_u,{\bf x}\rangle^*+a\langle {\bf s}_u,{\bf x}\rangle^*\quad(u\in V,~{\bf x}\in \mathbb{C}^d)
\end{equation*}
where $a=\frac{\omega+\overline{\omega}-2}{\omega^2+\overline{\omega}^2-\omega-\overline{\omega}}$ and observe that
\begin{equation}\label{eq-Fa-case1}
  F_u({\bf s}_v) = \left\{\begin{matrix}
                                               a(\lambda^2+\lambda)+\lambda^2 & \text{if~}u=v, \\
                                               1 & \text{if~}u\sim v, \\
                                               \omega & \text{if~}u\rightarrow v,\\
                                               \overline{\omega}  & \text{if~}u\leftarrow v, \\
                                               0  & \text{otherwise.}
                                             \end{matrix}\right.
\end{equation}
\begin{lemma}\label{lemma-linind-case1}
$F_1,F_2,\ldots,F_n$ are linearly independent if $\lambda\notin \{0,-1\}$.
\end{lemma}
\begin{proof}
  On the contrary,
  suppose that for some coefficients $\beta_{u}\in\mathbb{C}$, $u\in V$, 
  \begin{equation}\label{eq-lindepend-case1}
    \sum_{u=1}^n \beta_uF_u({\bf x})=0
  \end{equation}
  holds for all ${\bf x}\in \mathbb{C}^d$. 
  Put ${\bf b}:=(\beta_1,\beta_2,\ldots,\beta_n)^{\top}$.
  
  \begin{claim}\label{cl:case1-1}
  $\Big(\big(a(\lambda^2+\lambda)+\lambda^2\big)I+H\Big){\bf b}=0$.
  \end{claim}
  \begin{proof}
  By Eq.\ \eqref{eq-Fa-case1}, 
  substituting ${\bf x}={\bf s}_v$, 
  $v\in V$, into Eq.\ \eqref{eq-lindepend-case1} 
  implies that 
  \begin{equation*}
    \beta_v(a(\lambda^2+\lambda)+\lambda^2)+\sum_{u:~u\sim v}\beta_u + \omega\cdot \sum_{u':~u'\rightarrow v}\beta_{u'} +\overline{\omega}\cdot\sum_{u'':~u''\leftarrow v}\beta_{u''} = 0,
  \end{equation*}
  whence the claim follows.
  \end{proof}

\begin{claim}\label{cl:case1-2}
$(-\lambda I+H){\bf b}=0$.
\end{claim}
\begin{proof} 
Note that $a\neq 0$ and $a\neq -\frac{1}{2}$ by $|\omega|=1$.
Put $G_u(\mathbf{x},\mathbf{y}):=F_u({\bf x}+{\bf y})-F_u({\bf x}-{\bf y})$.
One can see that
\begin{equation*}\label{eq-diffF}
    G_u({\bf x},{\bf y}) = 4a\langle {\bf s}_u,{\bf x}\rangle\langle {\bf s}_u,{\bf y}\rangle+
    2\big(\langle {\bf s}_u,{\bf x}\rangle\langle {\bf s}_u,{\bf y}\rangle^*+
    \langle {\bf s}_u,{\bf x}\rangle^*\langle {\bf s}_u,{\bf y}\rangle\big)+
    2a\langle {\bf s}_u,{\bf y}\rangle^*.
  \end{equation*}

Now set ${\bf x}={\bf 1}$ and ${\bf y}={\bf s}_v$, $v\in V$. Taking into account that 
$\langle {\bf s}_u,{\bf x}\rangle=-1$ by Lemma \ref{lemma-split} and ${\bf 1}\in \mathcal{E}(\lambda)^{\perp}$, as $\lambda$ is a non-main eigenvalue, we obtain
\begin{align}
\nonumber
    G_u({\bf x},{\bf y}) &= -4a\langle {\bf s}_u,{\bf s}_v\rangle-
    2\big(\langle {\bf s}_u,{\bf s}_v\rangle^*+
    \langle {\bf s}_u,{\bf s}_v\rangle\big)+
    2a\langle {\bf s}_u,{\bf s}_v\rangle^*\\
    \label{eq-diff22}
&=(-4a-2)\langle {\bf s}_u,{\bf s}_v\rangle+(2a-2)\langle {\bf s}_u,{\bf s}_v\rangle^*,
\end{align}
and, further,  
\begin{align*}
  F_u(2{\bf y})-4F_u({\bf y})=&\big(4a\langle {\bf s}_u,{\bf y}\rangle^2+4\langle {\bf s}_u,{\bf y}\rangle\langle {\bf s}_u,{\bf y}\rangle^*+2a\langle {\bf s}_u,{\bf y}\rangle^*\big)\\
&-4\big(a\langle {\bf s}_u,{\bf y}\rangle^2+\langle {\bf s}_u,{\bf y}\rangle\langle {\bf s}_u,{\bf y}\rangle^*+a\langle {\bf s}_u,{\bf y}\rangle^*\big)\\
=&-2a\langle {\bf s}_u,{\bf y}\rangle^*=-2a\langle {\bf s}_u,{\bf s}_v\rangle^*,
\end{align*}
which together with Eq.\ \eqref{eq-diff22} implies that  
\begin{eqnarray}
\label{eq-diff12}
  \frac{1}{4a+2}\Big(G_u({\bf x},{\bf y})+\frac{a-1}{-a}\big(F_u(2{\bf y})-4F_u({\bf y})\big)\Big) &=& -\langle {\bf s}_u,{\bf s}_v\rangle \\
  \nonumber
   &=& \left\{\begin{matrix}
                                               -\lambda & \text{if~}u=v, \\
                                               1 & \text{if~}u\sim v, \\
                                               \omega & \text{if~}u\rightarrow v,\\
                                               \overline{\omega}  & \text{if~}u\leftarrow v, \\
                                               0  & \text{otherwise.}
                                             \end{matrix}\right.
\end{eqnarray}

It now follows from Eq.\ \eqref{eq-lindepend-case1} that
  \begin{equation}\label{eq-lindepend1-case1}
    \sum_{u=1}^n \frac{\beta_u}{4a+2}\Big(
F_u({\bf x}+{\bf y})-F_u({\bf x}-{\bf y})+\frac{a-1}{-a}\big(F_u(2{\bf y})-4F_u({\bf y})\big)\Big)=0
  \end{equation}
 holds for all ${\bf x},{\bf y}\in \mathbb{C}^d$,
  which by Eq.\ \eqref{eq-diff12} shows the claim. 
\end{proof}  
  
Finally, it follows from Claims \ref{cl:case1-1} and \ref{cl:case1-2} that
  \begin{equation*}
    \Big(\big(a(\lambda^2+\lambda)+\lambda^2\big)I+H\Big){\bf b}=(-\lambda I+H){\bf b}=0;
  \end{equation*}
  hence $a(\lambda^2+\lambda)+\lambda^2=-\lambda$, i.e., $\lambda\in \{0,-1\}$ (for $\omega\neq -1$ and $a\neq -1$), which is impossible,
  or $\beta_u=0$ for all $u\in V$, which shows the lemma.
\end{proof}

\begin{theorem}\label{th-case1}
$n\leq \frac{3d(d+1)}{2}$ holds.
\end{theorem}
\begin{proof}
$F_1,F_2,\ldots,F_n$ lie in the space of polynomials
in $x_kx_{\ell}$, $x_k\overline{x_{\ell}}$ and $x_k$ for ${\bf x}\in \mathbb{C}^d$, which has
dimension $d+{d\choose 2}+d^2+d=3d(d+1)/2$.
By Lemma \ref{lemma-linind-case1}, 
the statement follows.
\end{proof}

Our next result 
improves the bound from Theorem 
\ref{th-case1}; the proof is  similar to that of Theorem 
\ref{th-case1}, but with a much more involved argument.

\begin{theorem}\label{th-case1-improv}
With the above notation,
$n\leq \frac{3d(d+1)}{2}-1$ holds provided that $d\geq 2$.
\end{theorem}
\begin{proof}
Define a function $F$ by
\begin{equation*}
  F({\bf x})=a\langle {\bf 1},{\bf x}\rangle^2+\langle {\bf 1},{\bf x}\rangle\langle {\bf 1},{\bf x}\rangle^*+a\langle {\bf 1},{\bf x}\rangle^* \quad({\bf x}\in \mathbb{C}^d).
\end{equation*}
We shall prove that $F,F_1,\ldots,F_n$ are linearly independent.
Assume on the contrary that $F$ is written as a linear combination of $F_1,\ldots,F_n$ so that
\begin{equation}\label{eq-case1improv-lindep}
  F({\bf x})=\sum_{u=1}^n \beta_u F_u({\bf x})
\end{equation}
holds for all ${\bf x}\in \mathbb{C}^d$ and some coefficients  $\beta_{u}\in\mathbb{C}$, $u\in V$. Put 
${\bf b}:=(\beta_1,\ldots,\beta_n)^\top$.

\begin{claim}\label{cl:1imp1}
$\Big(\big(a(\lambda^2+\lambda)+\lambda^2\big)I+H\Big){\bf b}=\mathbf{1}$.  
\end{claim}
\begin{proof}
  Substituting ${\bf x}={\bf s}_v$, 
  $v\in V$, into Eq.\ \eqref{eq-case1improv-lindep} 
  gives  
\begin{align*}
1 &=\sum_{u}\beta_uF_u({\bf s}_v)\\
&=(a(\lambda^2+\lambda)+\lambda^2)\beta_v+\sum_{u:~u\sim v}\beta_u + \omega\cdot\sum_{u':~u'\rightarrow v}\beta_{u'} +\overline{\omega}\cdot\sum_{u'':~u''\leftarrow v}\beta_{u''}, \end{align*}
whence the claim follows.
\end{proof}

\begin{claim}\label{cl:1imp2}
$\frac{1}{4a+2}\big((-4a-4)\langle {\bf 1},{\bf 1}\rangle-4a+2\big){\bf 1}=(-\lambda I+H){\bf b}.$
\end{claim}
\begin{proof}
By Eq.\ \eqref{eq-case1improv-lindep}, 
we have the following equality (cf. 
Eq.\ \eqref{eq-lindepend1-case1}): 
\begin{align*}
&\frac{1}{4a+2}\Big(F({\bf x}+{\bf y})-F({\bf x}-{\bf y})+\frac{a-1}{-a}\big(F(2{\bf y})-4F({\bf y})\big)\Big)\\&=
\sum_{u=1}^n \frac{\beta_u}{4a+2}\Big(F_u({\bf x}+{\bf y})-F_u({\bf x}-{\bf y})+\frac{a-1}{-a}\big(F_u(2{\bf y})-4F_u({\bf y})\big)\Big),\end{align*}
which holds for all ${\bf x}, {\bf y}\in \mathbb{C}^d$.
As in Claim \ref{cl:case1-2}, substituting ${\bf x}={\bf 1},{\bf y}={\bf s}_v$, $v\in V$, gives 
\begin{align*}
\frac{1}{4a+2}\big((-4a-4)\langle {\bf 1},{\bf 1}\rangle-4a+2\big)=
-\sum_{u=1}^n \beta_u\langle {\bf s}_u,{\bf s}_v \rangle, \end{align*}
which shows the claim.
\end{proof}

\begin{claim}\label{cl:1imp3}
$\langle {\bf 1},{\bf x}\rangle\langle {\bf 1},{\bf y}\rangle^*=\beta\sum_{u=1}^n \langle {\bf s}_u,{\bf x}\rangle\langle {\bf s}_u,{\bf y}\rangle^*$ 
holds with some nonzero real scalar $\beta$.
\end{claim}
\begin{proof}
It follows from Claims \ref{cl:1imp1} and  \ref{cl:1imp2} that   
\begin{equation*}
 (a+1)(\lambda^2+\lambda){\bf b}=\frac{1}{4a+2}\big((4a+4)\langle {\bf 1},{\bf 1}\rangle+8a\big){\bf 1},
\end{equation*}
which implies ${\bf b}=\beta{\bf 1}$, i.e., $\beta_u=\beta$ for all $u\in V$ for some scalar $\beta$. Note that $\beta\in \mathbb{R}\setminus\{0\}$ as
 $a\neq-1$, $\lambda\not\in\{0,-1\}$ and 
$\langle {\bf 1},{\bf 1}\rangle,a,\lambda\in \mathbb{R}$.
Further, 
we proceed as follows.
\begin{enumerate}
\item[$(i)$] Put $F'({\bf x}):=\frac{1}{2}\big(F(2{\bf x})-2F({\bf x})\big)$ and $F'_u({\bf x}):=\frac{1}{2}\big(F_u(2{\bf x})-2F_u({\bf x})\big)$. 
One can see that 
\begin{eqnarray*}
F'({\bf x})&=&a\langle {\bf 1},{\bf x}\rangle^2+\langle {\bf 1},{\bf x}\rangle\langle {\bf 1},{\bf x}\rangle^*,\\
F'_u({\bf x})&=&a\langle {\bf s}_u,{\bf x}\rangle^2+\langle {\bf s}_u,{\bf x}\rangle\langle {\bf s}_u,{\bf x}\rangle^*.
\end{eqnarray*}
It then follows from Eq.\ \eqref{eq-case1improv-lindep} that 
\[
F'({\bf x})=\beta\sum_{u=1}^n F'_u({\bf x})
\]
holds for all $\mathbf{x}\in \mathbb{C}^d$. 

\item[$(ii)$] Put 
$G({\bf x},{\bf y}):=\frac{1}{2}\big(F'({\bf x}+{\bf y})-F'({\bf x}-{\bf y})\big)$ and  
$G_u({\bf x},{\bf y}):=\frac{1}{2} \big(F'_u({\bf x}+{\bf y})-F'_u({\bf x}-{\bf y})\big)$. One can see 
that 
\begin{eqnarray*}
  G({\bf x},{\bf y}) &=& 2a\langle {\bf 1},{\bf x}\rangle\langle {\bf 1},{\bf y}\rangle+\langle {\bf 1},{\bf x}\rangle\langle {\bf 1},{\bf y}\rangle^*+\langle {\bf 1},{\bf y}\rangle\langle {\bf 1},{\bf x}\rangle^*, \\
  G_u({\bf x},{\bf y}) &=& 2a\langle {\bf s}_u,{\bf x}\rangle\langle {\bf s}_u,{\bf y}\rangle+\langle {\bf s}_u,{\bf x}\rangle\langle {\bf s}_u,{\bf y}\rangle^*+\langle {\bf s}_u,{\bf y}\rangle\langle {\bf s}_u,{\bf x}\rangle^*.
\end{eqnarray*}
It follows from $(i)$ that 
\begin{equation*}
  G({\bf x},{\bf y}) = \beta\sum_{u=1}^n G_u({\bf x},{\bf y})
\end{equation*}
holds for all $\mathbf{x},\mathbf{y}\in \mathbb{C}^d$.

\item[$(iii)$] Finally, 
it follows from $(ii)$ that 
$\frac{1}{2}\big(G({\bf x},{\bf y})+{\bf i}G({\bf x},{\bf i}{\bf y})\big)=\frac{\beta}{2}\sum_{u=1}^n \big(G_u({\bf x},{\bf y})+{\bf i}G_u({\bf x},{\bf i}{\bf y})\big)$ 
holds for all $\mathbf{x},\mathbf{y}\in \mathbb{C}^d$, which in turn 
gives that 
\begin{align*}
\langle {\bf 1},{\bf x}\rangle\langle {\bf 1},{\bf y}\rangle^*=\beta\sum_{u=1}^n \langle {\bf s}_u,{\bf x}\rangle\langle {\bf s}_u,{\bf y}\rangle^*,
\end{align*}
\end{enumerate}
whence the claim follows.
\end{proof}

By Claim \ref{cl:1imp3}, 
we have that, 
for all $\mathbf{x},\mathbf{y}\in \mathbb{C}^d$, 
\begin{align*}
\beta^2\big|\sum_{u=1}^n \langle {\bf s}_u,{\bf x}\rangle\langle {\bf s}_u,{\bf y}\rangle^*\big|^2&
=\beta^2\Big(\sum_{u=1}^n \langle {\bf s}_u,{\bf x}\rangle\langle {\bf s}_u,{\bf y}\rangle^*\Big)
\Big(\sum_{u=1}^n \langle {\bf s}_u,{\bf x}\rangle\langle {\bf s}_u,{\bf y}\rangle^*\Big)^*\\
&=\langle {\bf 1},{\bf x}\rangle \langle {\bf 1},{\bf x}\rangle^* \langle {\bf 1},{\bf y}\rangle \langle {\bf 1},{\bf y}\rangle^* \\
&=\Big(\beta\sum_{u=1}^n \langle {\bf s}_u,{\bf x}\rangle\langle {\bf s}_u,{\bf x}\rangle^*\Big)
\Big(\beta\sum_{u=1}^n \langle {\bf s}_u,{\bf y}\rangle\langle {\bf s}_u,{\bf y}\rangle^*\Big),
\end{align*}
which implies that the Cauchy-Schwarz  inequality is attained.
Thus, for any ${\bf x},{\bf y}\in\mathbb{C}^d$, there exists a scalar  $\gamma=\gamma({\bf x},{\bf y})$ such that
$\langle {\bf s}_u,{\bf x}\rangle=\gamma \langle {\bf s}_u,{\bf y}\rangle$ for all $u\in V$.
Hence $\langle {\bf s}_u,{\bf x}-\gamma {\bf y}\rangle={\bf s}_u^* (\lambda I-C)^{-1}({\bf x}-\gamma {\bf y})=0$ for all $u$. This implies that 
$$
(C-\lambda I)(\lambda I-C)^{-1}({\bf x}-\gamma {\bf y})=0.
$$
Therefore, ${\bf x}=\gamma {\bf y}$ for all ${\bf x},{\bf y}\in \mathbb{C}^d$, which is true only when $d=1$. This contradicts the assumption  $d\geq 2$; thus, 
$F,F_1,\ldots,F_n$ are linearly independent. As in the proof of Theorem \ref{th-case1}, we conclude that $n+1\leq \frac{3d(d+1)}{2}$.
\end{proof}

\subsubsection{Case $\omega=\frac{-1\pm{\bf i}\sqrt{3}}{2}$}

Assume that $\omega=\frac{-1\pm\sqrt{-3}}{2}$.
Define functions $F_1,\ldots,F_n$ by
\begin{equation*}
  F_u({\bf x})=\langle {\bf s}_u,{\bf x}\rangle^2-\langle {\bf s}_u,{\bf x}\rangle+\langle {\bf s}_u,{\bf x}\rangle^*\quad(u\in V,~{\bf x}\in \mathbb{C}^d)
\end{equation*}
and observe that
\begin{equation}\label{eq-Fa-case2}
  F_u({\bf s}_v) = \left\{\begin{matrix}
                                               \lambda^2 & \text{if~}u=v, \\
                                               1 & \text{if~}u\sim v, \\
                                               \omega & \text{if~}u\rightarrow v,\\
                                               \overline{\omega}  & \text{if~}u\leftarrow v, \\
                                               0  & \text{otherwise.}
                                             \end{matrix}\right.
\end{equation}
\begin{lemma}\label{lemma-linind-ia-case2}
$F_1,F_2,\ldots,F_n$ are linearly independent if $\lambda\notin \{0,-1\}$.
\end{lemma}
\begin{proof}
  On the contrary,
  suppose that for some coefficients $\beta_{u}\in\mathbb{C}$, $u\in V$, 
  \begin{equation}\label{eq-lindepend1a-case2}
    \sum_{u=1}^n \beta_uF_u({\bf x})=0
  \end{equation}
  holds for all ${\bf x}\in \mathbb{C}^d$. Put 
${\bf b}:=(\beta_1,\ldots,\beta_n)^\top$. The rest of the proof is similar to that of Lemma \ref{lemma-linind-case1}. First, we show that 
  \begin{equation}\label{eq-case2H1}
      (\lambda^2I+H){\bf b}=0.
  \end{equation}
  Indeed, by Eq.\ \eqref{eq-Fa-case2}, 
  substituting ${\bf x}={\bf s}_v$, 
  $v\in V$, into Eq.\ \eqref{eq-lindepend1a-case2} 
  gives  
  \begin{equation*}
    \beta_v\lambda^2+\sum_{u:~u\sim v}\beta_u + \omega\cdot \sum_{u':~u'\rightarrow v}\beta_{u'} +\overline{\omega}\cdot\sum_{u'':~u''\leftarrow v}\beta_{u''} = 0,
  \end{equation*}
  whence Eq.\ \eqref{eq-case2H1} follows.
  
  Further, by Eq.\ \eqref{eq-lindepend1a-case2}, we have the following equality: 
  \begin{equation}\label{eq-lindepend2a-case2}
    \frac{1}{4}\sum_{u=1}^n \beta_u\big(F_u({\bf x}+{\bf y})-F_u({\bf x}-{\bf y})\big)=0,
  \end{equation}
  which holds for all ${\bf x},{\bf y}\in \mathbb{C}^d$. Observe that
\begin{equation}\label{eq-diffF2}
    F_u({\bf x}+{\bf y})-F_u({\bf x}-{\bf y}) = 4\langle {\bf s}_u,{\bf x}\rangle\langle {\bf s}_u,{\bf y}\rangle.
  \end{equation}

Substituting ${\bf x}={\bf 1}$ and ${\bf y}={\bf s}_v$,   $v\in V$,  
into Eq.\ \eqref{eq-diffF2} and taking into account that 
 $\langle {\bf s}_u,{\bf 1}\rangle=-1$ by Lemma \ref{lemma-split} and ${\bf 1}\in \mathcal{E}(\lambda)^{\perp}$, we obtain
\begin{align*}
    F_u({\bf 1}+{\bf s}_v)-F_u({\bf 1}-{\bf s}_v) &= -4\langle {\bf s}_u,{\bf s}_v\rangle,
  \end{align*}
  which implies $(-\lambda I+H){\bf b}=0$ by Eq.\ \eqref{eq-lindepend2a-case2}.
  Together with Eq.\ \eqref{eq-case2H1}, 
  this yields $(\lambda^2 I+H){\bf b}=(-\lambda I+H){\bf b}=0$, and
   thus $\lambda^2+\lambda=0$, i.e., $\lambda\in \{0,-1\}$, which is impossible,
  or $\beta_u=0$ for all $u\in V$, 
  which shows the lemma.
\end{proof}

\begin{theorem}\label{th-case2}
$n\leq \frac{d(d+5)}{2}$ holds.
\end{theorem}
\begin{proof}
$F_1,F_2,\ldots,F_n$ lie in the space of polynomials
in $x_kx_{\ell}$, $x_k$ and $\overline{x_k}$ for ${\bf x}\in \mathbb{C}^d$, which has
dimension ${d+1\choose 2}+d+d=d(d+5)/2$.
By Lemma \ref{lemma-linind-ia-case2}, we obtain $n\leq \frac{d(d+5)}{2}$.
\end{proof}

\section{Harmonic analysis on the complex unit sphere}\label{sect:harm}

In this section, we first briefly recall some basic theory of 
complex spherical designs and codes \cite{RS} and commutative association schemes \cite{BI}, 
which will be used later to prove Theorem \ref{theo-main}(ii).

\subsection{Spherical codes}
Let the complex Euclidean space $\mathbb{C}^d$ be equipped with the standard inner product ${\bf x}^*{\bf y}$ for ${\bf x},{\bf y}\in \mathbb{C}^d$.
Let $\Omega(d)$ denote the complex unit sphere in $\mathbb{C}^d$.
A {\bf complex spherical code} is a finite nonempty subset of $\Om(d)$.
For a complex spherical code $X$, 
define $A(X)$ by
\begin{align*}
A(X):=\{{\bf x}^*{\bf y} \mid {\bf x},{\bf y}\in X, {\bf x}\neq {\bf y}\},
\end{align*}
and the cardinality of $A(X)$ is called the {\bf degree} of $X$.

Let $\mathbb{N}$ denote the set of nonnegative integers.
A finite subset $\mathcal{S}$ of $\mathbb{N}^2$ is a {\bf lower set}
whenever the following condition is satisfied: if $(k,\ell)\in \mathcal{S}$, 
then so is $(i,j)$ for all $0\leq i\leq k$ and $0\leq j \leq \ell$.
A finite set $X$ in $\Omega(d)$ is an {\bf $\mathcal{S}$-code} if there exists a polynomial $F(x)=\sum_{(k,\ell)\in\mathcal{S}}a_{k,\ell}x^k\bar{x}^\ell$ with real coefficients such that $F(\alpha)=0$
for any $\alpha\in A(X)$ and $F(1)> 0$.

We denote by $\Hom(d,k,\ell)$ the vector space generated by
all homogeneous polynomials over $\mathbb{C}$ of degree $k$ in variables from $\{z_1,\ldots,z_d\}$ and of degree $\ell$ in variables from  $\{\bar{z}_1,\ldots,\bar{z}_d\}$.
The natural action of the  unitary group $U(d)$ on $\Hom(d,k,\ell)$  yields the irreducible decomposition of $\Hom(d,k,\ell)$ as follows:
\begin{align*}
\Hom(d,k,\ell)=\bigoplus_{m=0}^{\min\{k,\ell\}}\Harm(d,k-m,\ell-m),
\end{align*}
where $\Harm(d,k,\ell)\subseteq \Hom(d,k,\ell)$ is the kernel of the Laplace operator
${\displaystyle \sum_{i=1}^d\frac{\partial^2}{\partial z_i\partial\overline{z_i}}}$.
Further, let $m_{k,\ell}^d$ denote 
$\dim \Harm(d,k,\ell)$, that is   
\begin{align}
m_{k,\ell}^d&= \binom{d+k-1}{d-1}\binom{d+\ell-1}{d-1} - \binom{d+k-2}{d-1}\binom{d+\ell-2}{d-1}.\label{eq:dim}
\end{align}
An upper bound on the size of an $\mathcal{S}$-code is given by the following theorem.
\begin{theorem}[{\cite[Theorem~4.2 (ii)]{RS}}]\label{thm:42}
An $\mathcal{S}$-code $X$ in $\Omega(d)$, $d\geq 2$, satisfies
\begin{equation*}
  |X|\leq \sum_{(k,\ell)\in\mathcal{S}} m_{k,\ell}^d. 
\end{equation*}
\end{theorem}

An $\mathcal{S}$-code is {\bf tight} if equality holds in Theorem~\ref{thm:42}.
Tight codes are related to complex spherical designs.
For a finite lower set $\mathcal{T}$, a finite subset $X$ of $\Omega(d)$ is a {\bf complex spherical $\mathcal{T}$-design} if, for every polynomial $f \in \Hom(d,k,\ell)$ such that $(k,\ell)\in\mathcal{T}$, one has
\begin{equation}
\frac{1}{|X|} \sum_{\boldsymbol{z} \in X} f(\boldsymbol{z}) = \int_{\Omega(d)} f(\boldsymbol{z}) \mathrm{d} \boldsymbol{z},
\end{equation}
where $\mathrm{d}\boldsymbol{z}$ is the unique invariant Haar measure on $\Omega(d)$
normalized by $\int_{\Omega(d)} \mathrm{d} \boldsymbol{z}=1$.
As is stated in the following theorem, tight $\mathcal{S}$-codes are complex spherical $\mathcal{S}*\mathcal{S}$-designs,
where
\begin{center}
  $\mathcal{S}*\mathcal{S}:=\{(k+\ell',k'+\ell)\mid (k,\ell), (k',\ell') \in \mathcal{S}\}$.
\end{center}

Note that an $\mathcal{S}*\mathcal{S}$-design $X$ satisfies $|X|\geq \sum_{(k,\ell)\in\mathcal{S}}m_{k,\ell}^d$, and $X$ is {\bf tight} if it attains equality.

\begin{theorem}[{\cite[Theorem 5.4]{RS}}]\label{thm:tight}
Let $X$ be a finite  set in $\Omega(d)$ and $\mathcal{S}$ be a lower set.
Then the following are equivalent:
\begin{enumerate}
\item $X$ is a tight $\mathcal{S}$-code;
\item $X$ is a tight $\mathcal{S}*\mathcal{S}$-design;
\item $X$ is an $\mathcal{S}$-code and an $\mathcal{S} * \mathcal{S}$-design.
\end{enumerate}
\end{theorem}

Further, define an inner product for polynomials $f$ and $g$ on $\Om(d)$ as follows:
\[
\ip{f}{g} := \int_{\Om(d)} \overline{f(\boldsymbol{z})}g(\boldsymbol{z}) \dd \boldsymbol{z},
\]
with respect to which $\Harm(d,k,\ell)$ becomes orthogonal to $\Harm(d,k',\ell')$ whenever $(k,\ell) \neq (k',\ell')$.
For each $(k,\ell)\in\mathbb{N}^2$, fix an orthonormal basis $\{e_1,\ldots,e_{m^d_{k,\ell}}\}$ of the space $\Harm(d,k,\ell)$.

For a complex spherical code $X$ in $\Omega(d)$,
we define the characteristic matrix $H_{k,\ell}$ with
rows indexed by $X$, columns indexed by $\{1,2,\ldots,m^d_{k,\ell}\}$,
and entries given by
\begin{align*}
(H_{k,\ell})_{{\bf x},i}=e_i({\bf x})
\end{align*}
for ${\bf x}\in X$ and $i\in\{1,2,\ldots,m^d_{k,\ell}\}$.

For each $(k,\ell) \in \mathbb{N}^2$, we define the {\bf Jacobi polynomial} $g_{k,\ell}^d$ as follows:
\begin{equation}\label{eq-Jacobi}
g_{k,\ell}^d(x) := \frac{m_{k,\ell}^d(d-2)!k!\ell!}{(d+k-2)!(d+\ell-2)!} \sum_{r=0}^{\min\{k,\ell\}} (-1)^r \frac{(d+k+\ell-r-2)!}{r!(k-r)!(\ell-r)!}x^{k-r}\overline{x}^{\ell-r}.
\end{equation}

The essential property of the Jacobi polynomials is the following theorem, also known as Koornwinder's addition theorem.

\begin{theorem}[\cite{K}]\label{thm:addition}
Let $\{e_1,\ldots,e_{m^d_{k,\ell}}\}$ be an orthonormal basis for the space $\Harm(d,k,\ell)$. Then
\[
\sum_{i=1}^{m^d_{k,\ell}} \overline{e_i({\bf x})}e_i({\bf y}) = g_{k,\ell}^d({\bf x}^*{\bf y})
\]
holds for any ${\bf x},{\bf y} \in \Om(d)$.
\end{theorem}

Let $X$ be a complex $\mathcal{S}$-code with $A(X)=\{\alpha_1,\ldots,\alpha_s\}$, and set $\alpha_0:=1$.
For $0\leq i\leq s$, define a binary relation $R_i$ as the set of pairs
$({\bf x},{\bf y})\in X\times X$ such that ${\bf x}^*{\bf y}=\alpha_i$.
Tight codes possess a structure of commutative association schemes provided that $s=|\mathcal{S}|-1$ holds.

\begin{theorem}[{\cite[Theorem 6.1]{RS}}]\label{thm:S1}
Let $X$ be a tight $\mathcal{S}$-code with degree $s=|\mathcal{S}|-1$ for a lower set $\mathcal{S}$.
Then $X$ with the binary relations $R_0,R_1,\ldots,R_s$ defined as above is a commutative association scheme.
Moreover, its primitive idempotents are $\frac{1}{|X|}H_{k,\ell}H_{k,\ell}^*$, $(k,\ell)\in \mathcal{S}$.
\end{theorem}

We review the theory of association schemes in the next section.

\subsection{Association schemes}\label{ssect:as}

Let $X$ be a finite set and let $R_i$, $i\in\{0,1,\ldots,s\}$, be nonempty binary relations on $X$.
The {\bf adjacency matrix} $A_i$ of the relation $R_i$ is defined to be the $(0,1)$-matrix
with rows and columns indexed by $X$ such that $(A_i)_{xy}=1$ if $(x,y)\in R_i$ and $(A_i)_{xy}=0$ otherwise.
A pair $(X,\{R_i\}_{i=0}^s)$ is a {\bf commutative association scheme},
or simply an {\bf association scheme} if the following five conditions hold:
\begin{enumerate}
\item $A_0$ is the identity matrix.
\item $\sum_{i=0}^sA_i=J$, where $J$ is the all-one matrix.
\item for each $i \in\{0,1,\ldots,s\}$, 
there exists $i' \in\{0,1,\ldots,s\}$ such that $A_i^\top=A_{i'}$.
\item $A_iA_j=\sum_{\ell=0}^sp_{i,j}^{\ell}A_{\ell}$ for all $i,j\in\{0,1,\ldots,s\}$ (note that $p_{i,j}^{\ell}\in \mathbb{N}$).
\item $A_iA_j=A_jA_i$ for all $i,j$.
\end{enumerate}

The matrix algebra $\mathcal{A}$ generated by  $A_0,A_1,\ldots,A_s$ over $\mathbb{C}$ is called
the {\bf Bose-Mesner algebra} of the association scheme.
Since the  Bose-Mesner algebra is semisimple and commutative,
it has a unique set of primitive idempotents, which are denoted by $E_0:=\frac{1}{|X|}J,E_1,\ldots,E_s$ (see \cite[Theorem 3.1]{BI}). 
Since $\{E_0^\top,E_1^\top,\ldots,E_s^\top\}$ also forms the set of primitive idempotents,
for each $\ell\in \{0,1,\ldots,s\}$, 
there exists $\hat{\ell}\in \{0,1,\ldots,s\}$ such that $E_{\hat{\ell}}=E_{\ell}^\top$.
Note that $\hat{0}=0$.

Both sets of matrices $\{A_0,A_1,\ldots,A_s\}$ and $\{E_0,E_1,\ldots,E_s\}$ are bases for the Bose-Mesner algebra.
Therefore, there exist change of basis matrices $P$ and $Q$ defined as follows:
\[
A_i=\sum_{j=0}^sP_{ji}E_j,\quad E_j=\frac{1}{|X|}\sum_{i=0}^sQ_{ij}A_i.
\]
Then we have $P=\frac{1}{|X|}Q^{-1}$.
We call $P$ and $Q$ the {\bf (first) eigenmatrix} and the {\bf second eigenmatrix} of the association scheme, respectively.
For each $i\in\{0,1,\ldots,s\}$, $k_i:=P_{i0}$ and $m_i:=Q_{i0}$ are called
the $i$-{\bf th valency} and {\bf multiplicity}, respectively. Note that all 
$k_i, m_i\in \mathbb{N}$.

\subsection{The proof of Theorem \ref{theo-main}(ii)}\label{sect:codes}

\begin{lemma}\label{lem:ub}
Assume that $d\geq 2$.
Let $X$ be a subset of $\Omega(d)$ 
with $n=|X|$ and such that 
$$
\{x+{\bf i}y,x-{\bf i}y\} \subseteq A(X)\subseteq \{a,b,x+{\bf i}y,x-{\bf i}y\},
$$
where $a,b,x,y$ are real numbers with $a\neq b$ and $y\neq 0$. 
Then the following holds.
\begin{enumerate}
\item If $|A(X)|=2$, then $|X|\leq 2d+1$.
\item If $|A(X)|=3$, then $|X|\leq d^2+2d$.
\item If $|A(X)|=4$, then $|X|\leq \frac{1}{2}d(3d+5)-1$.
\end{enumerate}
\end{lemma}
\begin{proof}
(1) and (2) are shown in \cite{NS} and \cite{NS2}, respectively.

(3) First we show that $n\leq \frac{d(3d+5)}{2}$.
Let $\mathcal{S}:=\{(0,0),(1,0),(2,0),(0,1),(1,1)\}$.
Define a polynomial $F(z)$ in $z$ and $\bar{z}$ by
$
F(z)=\sum_{(k,\ell)\in\mathcal{S}}a_{k,\ell}g_{k,\ell}(z),
$
(with $g_{k,\ell}=g^d_{k,\ell}$, see Eq.\ \eqref{eq-Jacobi}), where
\begin{align*}
a_{0,0}&=y^2 (2 a b d+1)+(a-x) (x-b), \\
a_{1,0}&=x (a-x) (x-b)-y^2 (a+b+x), \\
a_{0,1}&=x (a-x) (b-x)-y^2 (a+b-x), \\
a_{1,1}&=\frac{(a-x) (x-b)+y^2}{d+1}, \\
a_{2,0}&=\frac{2 \left((a-x) (b-x)+y^2\right)}{d+1}.
\end{align*}

Then $F(\alpha)=0$ for any $\alpha\in A(X)$ and $F(1)=2 (1-a) (1-b) d y^2>1$ by $a< 1$ and $b<1$ (recall that $X\subseteq \Omega(d)$),
and $F(z)\in\text{Span}\{1,z,\bar{z},z^2,z\bar{z}\}$.
Therefore  $X$ is an $\mathcal{S}$-code, and by Theorem~\ref{thm:42}
, one has
\begin{equation}\label{ineq:code}
n=|X|\leq \sum_{(k,\ell)\in \mathcal{S}}m_{k,\ell}^d=\frac{d(3d+5)}{2}.
\end{equation}

Further, we show that there are no complex codes attaining the upper bound \eqref{ineq:code}.
Assume on the contrary that there exists such a code $X$.
Then $X$ carries a structure of commutative association scheme by Theorem~\ref{thm:S1}. More precisely, 
put $$\alpha_0=1,\alpha_1=a,\alpha_2=b,\alpha_3=x+{\bf i}y,\alpha_4=x-{\bf i}y,$$
and define
$$R_i=\{({\bf x},{\bf y})\mid {\bf x},{\bf y}\in X, {\bf x}^*{\bf y}=\alpha_i\}$$
for $i=0,1,\ldots,4$.
Then a pair $(X,\{R_i\}_{i=0}^4)$ forms a commutative association scheme with the second eigenmatrix $Q$ given as
\begin{align*}
Q=\begin{pmatrix}
1 & g_{1,0}(\alpha_0) & g_{0,1}(\alpha_0) & g_{1,1}(\alpha_0) & g_{2,0}(\alpha_0) \\
1 & g_{1,0}(\alpha_1) & g_{0,1}(\alpha_1) & g_{1,1}(\alpha_1) & g_{2,0}(\alpha_1) \\
1 & g_{1,0}(\alpha_2) & g_{0,1}(\alpha_2) & g_{1,1}(\alpha_2) & g_{2,0}(\alpha_2) \\
1 & g_{1,0}(\alpha_3) & g_{0,1}(\alpha_3) & g_{1,1}(\alpha_3) & g_{2,0}(\alpha_3) \\
1 & g_{1,0}(\alpha_4) & g_{0,1}(\alpha_4) & g_{1,1}(\alpha_4) & g_{2,0}(\alpha_4) \\
\end{pmatrix}.
\end{align*}

Since $g_{1,0}(\alpha_i)=\overline{g_{0,1}(\alpha_i)}$ for any $i$ and $g_{1,1}(1)\neq g_{1,0}(1)\neq g_{2,0}(1)$,
$E_1^\top=E_2$ holds. 
Thus, $E_3^\top=E_3$ and $E_4^\top=E_4$ hold, which implies in turn that the column corresponding to $E_4$ have only real numbers; these entries are:
\begin{align*}
g_{2,0}(\alpha_1)&=\frac{1}{2}a^2,\\
g_{2,0}(\alpha_2)&=\frac{1}{2}b^2,\\
g_{2,0}(\alpha_3)&=\frac{1}{2}(x+{\bf i}y)^2=\frac{1}{2}(x^2-y^2+2{\bf i}xy),\\
g_{2,0}(\alpha_4)&=\frac{1}{2}(x-{\bf i}y)^2=\frac{1}{2}(x^2-y^2-2{\bf i}xy);
\end{align*}
therefore, by $y\neq 0$, it follows that $x=0$.
Now the second eigenmatrix $Q$ becomes:
\begin{align*}
Q=\left(
\begin{array}{ccccc}
 1 & d & d & d^2-1 & \frac{1}{2} d (d+1) \\
 1 & a d & a d & (d+1) \left(a^2 d-1\right) & \frac{1}{2} a^2 d (d+1) \\
 1 & b d & b d & (d+1) \left(b^2 d-1\right) & \frac{1}{2} b^2 d (d+1) \\
 1 & {\bf i} d y & -{\bf i} d y & (d+1) \left(d y^2-1\right) & -\frac{1}{2} d (d+1) y^2 \\
 1 & -{\bf i} d y & {\bf i} d y & (d+1) \left(d y^2-1\right) & -\frac{1}{2} d (d+1) y^2 \\
\end{array}
\right).
\end{align*}

By $QP=|X|I$, we have that the row sums other than the first row are all equal to $0$,
which gives:
\begin{align*}
\frac{1}{2} d \left(3 a^2 (d+1)+4 a-2\right)&=0,\\
\frac{1}{2} d \left(3 b^2 (d+1)+4 b-2\right)&=0,\\
\frac{1}{2} d \left((d+1) y^2-2\right)&=0.
\end{align*}

By $a\neq b$, we obtain $(a,b)=(\frac{-2\pm\sqrt{6 d+10}}{3 (d+1)},\frac{-2\mp \sqrt{6 d+10}}{3 (d+1)})$ and
$y=\pm \frac{\sqrt{2}}{\sqrt{d+1}}$, namely $(\alpha_1,\alpha_2)=(\frac{-2\pm\sqrt{6 d+10}}{3 (d+1)},\frac{-2\mp \sqrt{6 d+10}}{3 (d+1)})$ and $\alpha_3=-\alpha_4=\pm {\bf i}\frac{\sqrt{2}}{\sqrt{d+1}}$.
Now we may assume, by suitably reordering  rows and columns, that $Q$ equals
\begin{align}\label{eq:q}
\left(
\begin{array}{ccccc}
 1 & d & d & d^2-1 & \frac{1}{2} d (d+1) \\
 1 & \frac{d \left(r-2\right)}{3 (d+1)} & \frac{d \left(r-2\right)}{3 (d+1)} & -\frac{d \left(3 d+4 r+4\right)+9}{9 (d+1)} & \frac{d \left(r-2\right)^2}{18 (d+1)} \\
 1 & -\frac{d \left(r+2\right)}{3 (d+1)} & -\frac{d \left(r+2\right)}{3 (d+1)} & \frac{d \left(-3 d+4 r-4\right)-9}{9 (d+1)} & \frac{d \left(r+2\right)^2}{18 (d+1)} \\
 1 & \frac{{\bf i} \sqrt{2} d}{\sqrt{d+1}} & -\frac{{\bf i} \sqrt{2} d}{\sqrt{d+1}} & d-1 & -d \\
 1 & -\frac{{\bf i} \sqrt{2} d}{\sqrt{d+1}} & \frac{{\bf i} \sqrt{2} d}{\sqrt{d+1}} & d-1 & -d
\end{array}
\right),
\end{align}
where $r=\sqrt{6 d+10}$. 

Calculating the first eigenmatrix by $P=|X|Q^{-1}$, we obtain
$$
k_1=\frac{1}{16} \left((3d-5) \sqrt{6 d+10}+9 d^2+12d-13\right),
$$
which, by $d\geq 2$, implies that $6d+10=4m^2$ for some positive integer $m\geq3$.
Substituting $d=\frac{2m^2-5}{3}$ back into Eq.\ \eqref{eq:q} gives
\begin{align*}
Q=\left(
\begin{array}{ccccc}
 1 & \frac{\left(2 m^2-5\right)}{3}  & \frac{\left(2 m^2-5\right)}{3}  & \frac{4\left(m^4-5 m^2+4\right)}{9}  & \frac{\left(m^2-1\right) \left(2 m^2-5\right)}{9}  \\
 1 & \frac{2 m^2-5}{3 (m+1)} & \frac{2 m^2-5}{3 (m+1)} & -\frac{2 (m-1) \left(m^2+6 m+8\right)}{9 (m+1)} & \frac{(m-1) \left(2 m^2-5\right)}{9 (m+1)} \\
 1 & \frac{5-2 m^2}{3 (m-1)} & \frac{5-2 m^2}{3 (m-1)} & -\frac{2 (m+1) \left(m^2-6 m+8\right)}{9 (m-1)} & \frac{(m+1) \left(2 m^2-5\right)}{9 (m-1)} \\
 1 & \frac{{\bf i} \left(2 m^2-5\right)}{\sqrt{3} \sqrt{m^2-1}} & -\frac{{\bf i} \left(2 m^2-5\right)}{\sqrt{3} \sqrt{m^2-1}} & \frac{2\left(m^2-4\right)}{3}  & \frac{\left(5-2 m^2\right)}{3}  \\
 1 & -\frac{{\bf i} \left(2 m^2-5\right)}{\sqrt{3} \sqrt{m^2-1}} & \frac{{\bf i} \left(2 m^2-5\right)}{\sqrt{3} \sqrt{m^2-1}} & \frac{2\left(m^2-4\right)}{3}  & \frac{\left(5-2 m^2\right)}{3}  \\
\end{array}
\right).
\end{align*}
Then it is a routine to show that
the intersection number $p_{1,1}^1$ (see Section \ref{ssect:as}),  
defined by 
\[
p_{i,j}^{\ell}:=|\{
\mathbf{z}\in X\mid (\mathbf{x},\mathbf{z})\in R_i, (\mathbf{z},\mathbf{y})\in R_j
\}|\quad\mbox{for any~}(\mathbf{x},\mathbf{y})\in R_{\ell},
\]
satisfies
\begin{align*}
32p_{1,1}^1=\frac{(m-2) (m+1)^3 (3 m-4)}{(m-1)}=3 m^4+2 m^3-11 m^2-14 m+\frac{8}{m-1}.
\end{align*}
Since $\frac{8}{m-1}$ must be an integer, it follows that $m=3,5,9$.
If $m=5,9$,  then $p_{1,1}^1=\frac{891}{16},\frac{20125}{32}$ respectively, a contradiction.
If $m=3$, then one can calculate that $p_{1,1}^2=\frac{25}{4}$, a contradiction.
Therefore, there are no complex codes attaining the bound in Eq.\ \eqref{ineq:code}, 
and the lemma follows. 
\end{proof}

\begin{lemma}\label{lem:ub1}
Let $X$ be a subset of the complex unit sphere in $\mathbb{C}^1$ with $n=|X|$.
Assume there exist real numbers $a,b,x,y$ with $a\neq b$ and $y\neq 0$ such that $
A(X)\subseteq\{a,b,x+{\bf i}y,x-{\bf i}y\}$.
Then $n\leq 4$, and equality holds if and only if $X=\{\pm1,\pm {\bf i}\}$ up to rotating.
\end{lemma}
\begin{proof}
By rotating $X$, we may assume that $1\in X$.
Then $X=A(X)\cup\{1\}$ holds.
Since there are exactly two real numbers $1,-1$ on the unit sphere in $\mathbb{C}^1$,
we have $A(X)\subseteq \{-1,x+{\bf i}y,x-{\bf i}y\}$.
Therefore, $|X|\leq 4$ and equality holds if and only if $x=0$ and $y=\pm1$, that is $X=\{\pm1,\pm {\bf i}\}$.
\end{proof}

Now we apply these results to digraphs.
To do so, we need the concept of main angles.
Let $\Delta$ be a digraph on $n$ vertices, 
and $H:=H_{\omega}(\Delta)$, for some $\omega\in\mathbb{C}\setminus \mathbb{R}$, $|\omega|=1$, be a Hermitian adjacency matrix of $\Delta$. Let $\lambda_1>\cdots>\lambda_s$ 
denote the distinct eigenvalues of $H$, and
$P_i$ be the orthogonal projection matrix onto the eigenspace of $H$ corresponding to the eigenvalue $\lambda_i$.
Then $P_i^*=P_i$, $\sum_{i=1}^s P_i=I$ and $P_i P_j=\delta_{i j}P_i$, where $\delta_{i j}$ denotes the Kronecker delta.
Define $\beta_i$ by
\[
\beta_i:= \frac{1}{\sqrt{n}} \sqrt{(P_i \cdot \boldsymbol{1})^*(P_i \cdot \boldsymbol{1})}.
\]
We call $\beta_i$ the {\bf main angle} of $\lambda_i$. Observe that $\beta_i=0$ if and only if $\lambda_i$ is non-main. 
By the definition of main angles, we have
\begin{align}\label{eq:01}
\sum_{i=1}^s \beta_i^2=1.
\end{align}

\begin{lemma}[\cite{NS}]\label{lem:char}
Let $K$ be a Hermitian matrix of size $n$ with $s$ distinct eigenvalues $\lambda_1>\cdots >\lambda_s$.
Let $\beta_i$ be the main angle of $\lambda_i$.
Then, for a real number $a$, one has
\[
P_{K+aJ}(x)=P_K(x)\Big( 1+a \sum_{i=1}^s  \frac{n \beta_i^2}{\lambda_i-x} \Big),
\]
where $P_M$ denotes the characteristic polynomial of a matrix $M$.
\end{lemma}
\begin{lemma}[\cite{NS}] \label{lem:interlace}
Let $K$ be a Hermitian matrix of size $n$,
and $M=K+a J$, where $a$ is a real number.
Let $\tau_{1}<\tau_{2}<\cdots<\tau_{r}$ be the distinct main
eigenvalues of $K$, and $\beta_i$ the main angle of $\tau_i$.
Let $\mu_{1}<\mu_{2}<\cdots<\mu_{s}$ be the distinct
main eigenvalues of $M$.
Then $r=s$ holds, and
\begin{equation} \label{eq:f(x)}
\prod_{i=1}^r (\mu_{i}-x)= \Big(1+a \sum_{j=1}^r \frac{n \beta_{j}^2}{\tau_{j}-x}\Big)\prod_{i=1}^r(\tau_{i}-x)
.
\end{equation}
 Moreover, if $a>0$, then
$\tau_{1}<\mu_{1}<\tau_{2}< \cdots <\tau_{r}<\mu_{r}$,
and if $a<0$, then
 $\mu_{1}<\tau_{1}<\mu_{2} <\cdots <\mu_{r}<\tau_{r}$.
\end{lemma}

\begin{lemma}[\cite{NS}]\label{lem:non-regular}
Let $H$ be a Hermitian adjacency matrix of order $n$ with $s$ distinct eigenvalues $\lambda_1>\lambda_2>\cdots >\lambda_s$
of respective multiplicities $m_1,m_2,\ldots,m_s$.
Set $d_i=n-m_i$ for $i\in\{1,2,\ldots,s\}$.
The following statements hold. 
\begin{enumerate}
\item Put $c_s=-1/(\sum_{j=1}^{s-1}\frac{n\beta_j^2}{\lambda_j-\lambda_s})$. If $\lambda_s$ is non-main, then $\rank(H+c_sJ-\lambda_s I)=d_s-1$ and $I-\frac{1}{\lambda_s}H-\frac{c_s}{\lambda_s}J$ is positive semidefinite.

\item Put $c_1=-1/(\sum_{j=2}^{s}\frac{n\beta_j^2}{\lambda_j-\lambda_1})$. If $\lambda_1$ is non-main,  then $\rank(\lambda_1 I-H-c_1J)=d_1-1$ and $I-\frac{1}{\lambda_1}H-\frac{c_1}{\lambda_1}J$ is positive semidefinite.

\item Put $c_{s-1}=-1/(\sum_{j=1}^{s-2}  \frac{n \beta_{j}^2}{\lambda_{j}-\lambda_{s-1}}+\frac{n \beta_{s}^2}{\lambda_{s}-\lambda_{s-1}})$. If $\lambda_s$ is main, $\lambda_{s-1}$ is non-main, $\lambda_{s-1}<0$, $m_s=1$, and $c_{s-1}<0$, then $\rank(H+c_{s-1}J-\lambda_{s-1} I)=d_{s-1}-1$ and $I-\frac{1}{\lambda_{s-1}}H-\frac{c_{s-1}}{\lambda_{s-1}}J$ is positive semidefinite.

\item Put $c_2=-1/(\frac{n \beta_{s}^2}{\lambda_{1}-\lambda_{2}}+\sum_{j=3}^{s}  \frac{n \beta_{j}^2}{\lambda_{j}-\lambda_{2}})$. If $\lambda_1$ is main, $\lambda_{2}$ is non-main, $\lambda_2>0$, $m_1=1$, and $c_2<0$, then $\rank(\lambda_2 I-H-c_2J)=d_2-1$ and $I-\frac{1}{\lambda_2}H-\frac{c_2}{\lambda_2}J$ is positive semidefinite.
\end{enumerate}
\end{lemma}
\begin{proof}
Let $\lambda_{i_1}>\cdots>\lambda_{i_t}$ be the main eigenvalues of $H$ and set
$[s]=\{1,2,\ldots,s\}$, $N=\{i_1,\ldots,i_t\}$.
Then, by Lemma~\ref{lem:char},
\begin{align*}
P_{H+aJ}(x)&=\prod_{j\in [s]\setminus N}(\lambda_j-x)^{m_j}\left(\prod_{j\in N}(\lambda_j-x)^{m_j}+a\sum_{i\in N}n\beta_{i}^2\prod_{j\in N}(\lambda_j-x)^{m_j-\delta_{i,j}} \right)
\end{align*}
holds. Define a polynomial
$$f(x,a):=\prod_{j\in N}(\lambda_j-x)^{m_j}+a\sum_{i\in N}n\beta_{i}^2\prod_{j\in N}(\lambda_j-x)^{m_j-\delta_{i,j}}. $$

(1) Take $a=c_s$. It is easy to see that $f(\lambda_s,c_s)=0$,
which implies that $H+c_s J$ has an eigenvalue $\lambda_s$ of multiplicity $m_s+1$.
By Lemma~\ref{lem:interlace}, we have that the least eigenvalue of $H+c_s J$ is $\lambda_s$.
Thus, $H+c_s J-\lambda_s I$ is positive semidefinite, which shows that
$I-\frac{1}{\lambda_s}H-\frac{c_s}{\lambda_s}J$ is positive semidefinite, as $\lambda_s<0$.
(2) can be proven in the same manner.

(3) Take $a=c_{s-1}$. It is easy to see that $f(\lambda_{s-1},c_{s-1})=0$,
which implies that $H+c_{s-1} J$ has an eigenvalue $\lambda_{s-1}$ of multiplicity $m_{s-1}+1$.
By Lemma~\ref{lem:interlace} and the assumption that $c_{s-1}<0$, we have that the least eigenvalue
of $H+c_{s-1} J$ is $\lambda_{s-1}$. Thus,  $H+c_{s-1} J-\lambda_{s-1} I$ is positive semidefinite,
which shows that $I-\frac{1}{\lambda_{s-1}}H-\frac{c_{s-1}}{\lambda_{s-1}}J$ is positive semidefinite because $\lambda_{s-1}<0$.
(4) can be proven in the same manner.
\end{proof}

The following theorem shows part (ii) of Theorem \ref{theo-main}.

\begin{theorem}\label{thm:non-regular}
With the notation as in Lemma \ref{lem:non-regular}, the following holds.
\begin{enumerate}
\item For $i\in\{1,s\}$, if $\lambda_i\not\in\{-1,0\}$, then $n\leq \frac{d_i(3d_i+5)}{2}-1$ holds.
\item For $i\in\{1,s\}$, if $\lambda_i$ is non-main and $\lambda_i\not\in\{-1,0\}$, then $n\leq \frac{(d_i-1)(3d_i+2)}{2}-1$ holds.
\item If $\lambda_{1}$ is main, $\lambda_2>0$ is non-main, and $\sum_{j=1,j\neq 2}^{s}  \frac{n \beta_{j}^2}{\lambda_{j}-\lambda_{2}}<0$, then $n\leq \frac{(d_2-1)(3d_2+2)}{2}-1$ holds.
\item If $\lambda_{s}$ is main, $\lambda_{s-1}<0$ is non-main, $\lambda_{s-1}\neq-1$, and $\sum_{j=1,j\neq s-1}^{s}  \frac{n \beta_{j}^2}{\lambda_{j}-\lambda_{s-1}}<0$, then $n\leq \frac{(d_{s-1}-1)(3d_{s-1}+2)}{2}-1$ holds.
\end{enumerate}
\end{theorem}
\begin{proof}
(1)  Note that $\lambda_1\neq0,\lambda_s\neq0$ by assumption.
When $i=1$, the matrix $\lambda_1 I-H$ is a positive semidefinite Hermitian matrix of rank $d_1$.
Similarly, when $i=s$, the matrix $H-\lambda_s I$ is a positive semidefinite Hermitian matrix of rank $d_s$.
In both cases, we regard $G:=I-\frac{1}{\lambda_i}H$ as the Gram matrix of a finite set $X$ of $\Omega(d)$.
Then the inner products between two distinct points in $X$ are the off-diagonal entries of $G$, so they are in
$$\left\{0,-\frac{1}{\lambda_i},\frac{\omega}{\lambda_i},\frac{\overline{\omega}}{\lambda_i}\right\},$$
whence, by Lemma~\ref{lem:ub}, we have $n\leq \frac{d(3d+5)}{2}-1$.

For (2), take $j=i$. For (3), take $j=2$. For (4), take $j=s-1$. 
In each case,
$I-\frac{1}{\lambda_j}H-\frac{c_j}{\lambda_j}J$ is positive semidefinite.
We now regard $G:=I-\frac{1}{\lambda_j+c_j}H+\frac{c_j}{\lambda_j+c_j}(J-I)$ as the Gram matrix of a finite set $X$ of $\Omega(d-1)$.
Then the inner products between two distinct points in $X$ are the off-diagonal entries of $G$, so they are in
$$\left\{\frac{-1+c_i}{\lambda_j+c_j},\frac{-\omega+c_i}{\lambda_j+c_j},\frac{-\overline{\omega}+c_i}{\lambda_j+c_j}\right\},$$
whence, by Lemma~\ref{lem:ub}, we have $n\leq \frac{(d-1)(3d+2)}{2}-1$.
\end{proof}

\begin{corollary}\label{cor:upper}
With the above notation,
assume that $H$ has an eigenvector ${\bf 1}$ with eigenvalue $k$. Then:
\begin{enumerate}
\item $n\leq \frac{d_1(3d_1+5)}{2}-1$, and if $\lambda_1$ is non-main, then $n\leq \frac{(d_1-1)(3d_1+2)}{2}-1$ holds;
\item $n\leq \frac{(d_s-1)(3d_s+2)}{2}-1$ holds.
\end{enumerate}
\end{corollary}

\section{Upper bounds for $\mathcal{S}$-codes in $\mathbb{C}^{p,q}$}\label{sect:Cpq}

In this section, we present a complex counter part of the theory of distance sets in real hyperbolic spaces 
developed by Blokhuis \cite{Bl}. It is then used to prove Theorem \ref{theo-main-2}, 
see Theorem \ref{thm:main2} below.

\subsection{Harmonic polynomials}
Let $B$ be a bilinear form on $\mathbb{C}^d$ of inertia $(p,q)$ with $p+q=d$, defined by
\begin{align*}
B(\mathbf{x},\mathbf{y})=\overline{x_1}y_1+\cdots +\overline{x_p}y_p -\overline{x_{p+1}}y_{p+1}-\cdots-\overline{x_{d}}y_{d},
\end{align*}
where $\mathbf{x}=(x_1,\ldots,x_d),\mathbf{y}=(y_1,\ldots,y_d)$. 
The complex Euclidean space $\mathbb{C}^d$ equipped with the bilinear form $B$ is denoted by $\mathbb{C}^{p,q}$.

Let $\mathcal{R}=\mathbb{C}[z_1,\ldots,z_d,\overline{z_1},\ldots,\overline{z_d}]$ denote the algebra of polynomial functions 
in $z_1,\ldots,z_d$, $\overline{z_1},\ldots,\overline{z_d}$ on $\mathbb{C}^d$,
$\Hom(d,k,\ell)$ denote the space of the homogeneous polynomials of degree $k$ with respect to $z_1,\ldots,z_d$ 
and of degree $\ell$ with respect to $\overline{z_1},\ldots,\overline{z_d}$. 
Set 
\[
\beta:=\beta(\mathbf{z})=B(\mathbf{z},\mathbf{z})=z_1\overline{z_1}+\cdots+z_p\overline{z_p}-z_{p+1}\overline{z_{p+1}}-\cdots-z_d\overline{z_d},
\] 
and note that $\beta\in \Hom(d,1,1)$.

Write  $\partial_i=\frac{\partial}{\partial_{z_i}},\overline{\partial_i}=\frac{\partial}{\partial_{\overline{z_i}}}$.
For a given polynomial $f(z_1,\ldots,z_d)\in \mathcal{R}$, associate $f$ with the formal 
differential operator $f(\partial):=f(\partial_1,\ldots,\partial_p,-\partial_{p+1},\ldots,-\partial_d)$, 
where $\partial=(\partial_1,\ldots,\partial_p,-\partial_{p+1},\ldots,-\partial_d)$.
In particular, the {\bf Laplace operator} $\partial_\beta$ associated with the bilinear form $B$ is  defined as 
\begin{align*}
\partial_\beta:=\beta(\partial)=\frac{\partial^2}{\partial_1\overline{\partial_1}}+\cdots+\frac{\partial^2}{\partial_p\overline{\partial_p}}-\frac{\partial^2}{\partial_{p+1}\overline{\partial_{p+1}}}-\cdots-\frac{\partial^2}{\partial_{d}\overline{\partial_{d}}}.
\end{align*}

For $f,g \in \mathcal{R}$, we define an indefinite inner product $\langle f,g\rangle$ by
\begin{align*}
\langle f,g\rangle=(\underline{f}(\partial) g)(0),
\end{align*}
where $\underline{f}$ is defined to be the polynomial obtained by taking complex conjugate of the coefficients of $f$. 
Then the monomials $z^a\overline{z}^b:=z_1^{a_1}\cdots z_d^{a_d}\overline{z_1}^{b_1}\cdots \overline{z_d}^{b_d}$ 
form an orthogonal basis of $\mathcal{R}$ with the property that
\begin{align}
\langle z^a\overline{z}^b,z^a\overline{z}^b\rangle = (-1)^{\sum_{i=p+1}^d a_i+\sum_{i=p+1}^d b_i} \prod_{i=1}^d a_i! b_i!. \label{eq:mono}
\end{align}
In view of Eq.\ \eqref{eq:mono}, we obtain $\overline{\langle f,g\rangle}=\langle g,f\rangle$ and in turn 
\begin{align}
\langle f,gh\rangle=\overline{\langle gh, f\rangle}=\overline{(\underline{g}(\partial)\underline{h}(\partial)f)(0)}=\overline{\langle h, \underline{g}(\partial)f\rangle}=\langle \underline{g}(\partial)f,h\rangle. \label{eq:adj} 
\end{align}

Define $\Harm_B(d,k,\ell)$ to be the space of homogeneous polynomials $f$ of degree $k$ 
with respect to $z_1,\ldots,z_d$ and of degree $\ell$ with respect to $\overline{z_1},\ldots,\overline{z_d}$ 
satisfying $\partial_\beta f=0$, i.e.:
\begin{align*}
\Harm_B(d,k,\ell)=\Hom(d,k,\ell)\cap \ker \partial_\beta.
\end{align*}
Then Eq.\ \eqref{eq:mono} leads to the following decomposition:
\begin{align*}
\Hom(d,k,\ell)=\Hom^{+}(d,k,\ell)+\Hom^{-}(d,k,\ell),
\end{align*}
where
\begin{align*}
\Hom^{+}(d,k,\ell)&=\text{span}\{z^a\overline{z}^b\in \Hom(d,k,\ell) \mid \sum_{i=p+1}^d a_i+\sum_{i=p+1}^d b_i\equiv 0 \pmod{2}\},\\
\Hom^{-}(d,k,\ell)&=\text{span}\{z^a\overline{z}^b\in \Hom(d,k,\ell) \mid \sum_{i=p+1}^d a_i+\sum_{i=p+1}^d b_i\equiv 1 \pmod{2}\}.
\end{align*}
Then the space $\Harm_B(d,k,\ell)$ is decomposed into $\Harm_B^{+}(d,k,\ell)$ and $\Harm_B^{-}(d,k,\ell)$, where
\begin{align*}
\Harm_B^{+}(d,k,\ell)&=\Hom^+(d,k,\ell)\cap \Harm_B(d,k,\ell),\\
\Harm_B^{-}(d,k,\ell)&=\Hom^-(d,k,\ell)\cap \Harm_B(d,k,\ell).
\end{align*}

The dimensions of these spaces are determined in the following theorem. 

\begin{theorem}\label{thm:add-dim}
The following holds.
\begin{align*}
(i)& \dim \Hom(d,k,\ell)=\binom{d+k-1}{k}\binom{d+\ell-1}{\ell}.\\
(ii)& \dim \Harm_B(d,k,\ell)=\binom{d+k-1}{k}\binom{d+\ell-1}{\ell}-\binom{d+k-2}{k-1}\binom{d+\ell-2}{\ell-1}.\\
(iii)& \dim \Hom^{+}(d,k,\ell)=\sum_{i=0}^{\lfloor k/2\rfloor}\sum_{j=0}^
  {\lfloor \ell/2\rfloor}\binom{p+k-2i-1}{k-2i}\binom{q+2i-1}{2i}\binom{p+\ell-2j-1}{\ell-2j}\times \\
&\binom{q+2j-1}{2j}+\sum_{i=0}^{\lfloor (k-1)/2\rfloor}\sum_{j=0}^{\lfloor (\ell-1)/2\rfloor}\binom{p+k-2i-2}{k-2i-1}\binom{q+2i}{2i+1}\binom{q+2j}{2j+1}\times\\
&\binom{p+\ell-2j-2}{\ell-2j-1}.\\
(iv)& \dim \Hom^{-}(d,k,\ell)=\sum_{i=0}^{\lfloor k/2\rfloor}\sum_{j=0}^{\lfloor (\ell-1)/2\rfloor}\binom{p+k-2i-1}{k-2i}\binom{q+2i-1}{2i}\binom{q+2j}{2j+1}\times \\ \displaybreak[0]
& \binom{p+\ell-2j-2}{\ell-2j-1} +\sum_{i=0}^{\lfloor (k-1)/2\rfloor}\sum_{j=0}^{\lfloor \ell/2\rfloor}\binom{p+k-2i-2}{k-2i-1}\binom{q+2i}{2i+1}\binom{q+2j-1}{2j}\times\\
& \binom{p+\ell-2j-1}{\ell-2j}.\\
(v)& \dim \Harm_B^+(d,k,\ell)=\dim \Hom^{+}(d,k,\ell)-\dim \Hom^{+}(d,k-1,\ell-1). \\
(vi)& \dim \Harm_B^-(d,k,\ell)=\dim \Hom^{-}(d,k,\ell)-\dim \Hom^{-}(d,k-1,\ell-1).
\end{align*}
\end{theorem}
\begin{proof}
$(i)$ is easy to see, and $(ii)$ follows from the decomposition 
\begin{align}
\Hom(d,k,\ell)=\Harm_B(d,k,\ell)\oplus \beta\Hom(d,k-1,\ell-1), \label{eq:hom0}
\end{align}
which can be proved as follows. 
It is clear that $\Hom(d,k,\ell)\supseteq\Harm_B(d,k,\ell)+ \beta\Hom(d,k-1,\ell-1)$. 
For $f\in\Harm_B(d,k,\ell)$ and $g\in \Hom(d,k-1,\ell-1)$, it follows from Eq.\ \eqref{eq:adj} that 
\begin{align*}
\langle f,\beta g\rangle &=\langle \partial_\beta f, g\rangle=0, 
\end{align*}
which implies that $\Harm_B(d,k,\ell)\perp \beta\Hom(d,k-1,\ell-1)$. 
Therefore,  
\begin{align}
\Hom(d,k,\ell)\supseteq\Harm_B(d,k,\ell)\oplus \beta\Hom(d,k-1,\ell-1). \label{eq:hom1}
\end{align} 
Further, since the kernel of the linear map $\partial_\beta$ is $\Harm_B(d,k,\ell)$, it follows that 
\begin{align}
\nonumber
\dim \Hom(d,k,\ell)&=\dim \Harm_B(d,k,\ell)+\dim \textrm{Im}
(\partial_\beta)\\
&\leq \dim \Harm_B(d,k,\ell)+\dim \Hom(d,k-1,\ell-1). \label{eq:hom2}
\end{align} 
From Eqs. \eqref{eq:hom1} and \eqref{eq:hom2}, we conclude that Eq.\ \eqref{eq:hom0} follows.

For $z^a\overline{z}^b=z_1^{a_1}\cdots z_d^{a_d} \overline{z_1}^{b_1}\cdots \overline{z_d}^{b_d}$, 
write  $z^{a_+}=z_1^{a_1}\cdots z_p^{a_p}$, $z^{a_-}=z_{p+1}^{a_{p+1}}\cdots z_d^{a_d}$ and $\overline{z}^{b_+}=\overline{z_1}^{b_1}\cdots \overline{z_p}^{b_p}$, $\overline{z}^{b_-}=\overline{z_{p+1}}^{b_{p+1}}\cdots \overline{z_d}^{b_d}$. 
Then $z^a\overline{z}^b\in\Hom^+(d,k,\ell)$ if and only if one of the following holds:
\begin{itemize}
\item $z^{a_+}\overline{z}^{b_+}\in\Hom^+(p,k-2i-1,\ell-2j-1)$ and $z^{a_-}\overline{z}^{b_-}\in\Hom^+(q,2i+1,2j+1)$ for some $i,j$.
\item $z^{a_+}\overline{z}^{b_+}\in\Hom^+(p,k-2i,\ell-2j)$ and $z^{a_-}\overline{z}^{b_-}\in\Hom^+(q,2i,2j)$ for some $i,j$.
\end{itemize}
Therefore, it follows that
\begin{align*}
\dim \Hom^+(d,k,\ell) &=\sum_{i,j}\dim \Hom(p,k-2i-1,\ell-2j-1) \dim \Hom(q,2i+1,2j+1)\\
&+\sum_{i,j}\dim \Hom(p,k-2i,\ell-2j) \dim \Hom(q,2i,2j),
\end{align*}
which yields $(iii)$. 

The proof of $(iv)$ is similar. The result for $(v)$ and $(vi)$ follows from the decomposition 
\[
\Hom^{\epsilon}(d,k,\ell)=\Harm_B^{\epsilon}(d,k,\ell)\oplus \beta\Hom^{\epsilon}(d,k-1,\ell-1),
\] 
where $\epsilon\in \{+,-\}$.
\end{proof}

Further, we denote
\begin{align*}
\mu_{k,\ell}&:=\mu_{k,\ell}(p,q)=\dim \Harm_B^+(d,k,\ell),\\
\nu_{k,\ell}&:=\nu_{k,\ell}(p,q)=\dim \Harm_B^-(d,k,\ell).
\end{align*}

\subsection{Addition formula}
Given $\mathbf{x}\in \mathbb{C}^{p,q}$, the map $f\mapsto f(\mathbf{x})$ for 
$f\in\Harm_B(d,k,\ell)$ defines a function on the space $\Harm_B(d,k,\ell)$.
Then, by the Riesz representation theorem for a finite-dimensional nondegenerate bilinear form space, there exists a unique polynomial 
$r_\mathbf{x}\in\Harm_B(d,k,\ell)$ such that
\begin{align}\label{eq:add-repro}
\langle r_{\mathbf{x}},f\rangle= f(\mathbf{x}) \text{ for any }f\in\Harm_B(d,k,\ell).
\end{align}

Consider an ``orthonormal" basis $\{f_{k,\ell,i},g_{k,\ell,j} \mid i=1,\ldots,\mu_{k,\ell},j=1,\ldots,\nu_{k,\ell}\}$ 
of $\Harm_B(d,k,\ell)$ such that
\begin{align*}
\langle f_{k,\ell,i},f_{k,\ell,j}\rangle&=\delta_{ij},\\
\langle g_{k,\ell,i},g_{k,\ell,j}\rangle&=-\delta_{ij},\\
\langle f_{k,\ell,i},g_{k,\ell,j}\rangle&=0.
\end{align*}
Then the harmonic polynomial $r_{\mathbf{x}}$ is written as
\begin{align*}
r_{\mathbf{x}}=\sum_{i=1}^{\mu_{k,\ell}} \overline{\langle r_\mathbf{x},f_{k,\ell,i} \rangle} f_{k,\ell,i}-
\sum_{j=1}^{\nu_{k,\ell}} \overline{\langle r_\mathbf{x},g_{k,\ell,j} \rangle} g_{k,\ell,j},
\end{align*}
and, by Eq.\ \eqref{eq:add-repro},
\[
r_{\mathbf{x}}(\mathbf{y})=\sum_{i=1}^{\mu_{k,\ell}} \overline{f_{k,\ell,i}(\mathbf{x})} f_{k,\ell,i}(\mathbf{y})-
\sum_{j=1}^{\nu_{k,\ell}} \overline{g_{k,\ell,j}(\mathbf{x})} g_{k,\ell,j}(\mathbf{y}).
\]
Now we find out an explicit formula for  $r_{\mathbf{x}}(\mathbf{y})$. 

\begin{lemma}\label{lem:add-kl}
For $f(\mathbf{z}) \in \Hom(d,k,\ell)$, 
one has 
$\langle B(\mathbf{a},\mathbf{z})^k \overline{B(\mathbf{a},\mathbf{z})}^\ell , f(\mathbf{z})\rangle=k!\ell! f(\mathbf{a})$.
\end{lemma}
\begin{proof}
It suffices to show the result in the cases $k=0$ or $\ell=0$.
Here we consider the case $\ell=0$ by using  induction on $k$.
For $k=1$, write $\mathbf{a}=(a_1,\ldots,a_d)$ and $f=b_1z_1+\cdots+b_d z_d$.
Then
\begin{align*}
\langle B(\mathbf{a},\mathbf{z}), f \rangle &=\langle \sum_{i=1}^p \overline{a_i}z_i-\sum_{i=p+1}^d \overline{a_i}z_i, 
b_1z_1+\cdots+b_d z_d \rangle \\
&=\Big(\big(\sum_{i=1}^p a_i \partial_i-\sum_{i=p+1}^d a_i (-\partial_i)\big)\big(b_1z_1+\cdots+b_d z_d\big)\Big)(0)\\
&=\sum_{i=1}^p a_i b_i+\sum_{i=p+1}^d a_i b_i=f(\mathbf{a}).
\end{align*}
Suppose that $k>1$. 
By using $\langle gh,f \rangle=\langle g,\underline{h}(\partial)f \rangle$ and 
$\sum_{i=1}^{d} z_i \partial_i f(\mathbf{z})=kf(\mathbf{z})$, 
we have
\begin{align*}
\langle B(\mathbf{a},\mathbf{z})^k, f \rangle &=\langle B(\mathbf{a},\mathbf{z})^{k-1}, 
\big(\sum_{i=1}^p a_i \partial_i-\sum_{i=p+1}^d a_i (-\partial_i)\big)f \rangle\\
&=(k-1)! \big(\sum_{i=1}^d a_i \partial_i f\big)(\mathbf{a})\\
&=(k-1)!kf(\mathbf{a})\\
&=k!f(\mathbf{a}), 
\end{align*}
which completes the proof. \qedhere
\end{proof}

\begin{lemma}\label{lem:add-rec}
For any $f\in \Hom(d,k,\ell)$, the following holds:
\[
\partial_\beta(\beta^i \partial_\beta^i f)=\beta^i \partial_\beta^{i+1} f+i(d+k+\ell-i-1)\beta^{i-1} \partial_\beta^i f.
\]
\end{lemma}
\begin{proof}
By using 
$\partial_{\beta}(gh)=(\partial_\beta g)h+g(\partial_\beta h)+\sum_{i=1}^d \big((\partial_i g)(\overline{\partial_i}h)+
(\overline{\partial_i} g)(\partial_i h)\big)$,  
we obtain
\begin{align}\label{add-partial}
\partial_\beta(\beta^i \partial_\beta^i f)=(\partial_\beta\beta^i) (\partial_\beta^i f)+\beta^i \partial_\beta^{i+1} f+
\sum_{j=1}^d \big((\partial_j \beta^i)\overline{\partial_j}\partial_\beta^i f+(\overline{\partial_j}\beta^i)\partial_j \partial_\beta^i f\big).
\end{align}

Let us calculate the first and third terms in the right-hand side of Eq.\ \eqref{add-partial}.
Regarding the first term in Eq.\ \eqref{add-partial}, we have 
\begin{align*}
\partial_\beta\beta^i&=(\sum_{j=1}^p \partial_j\overline{\partial_j}- \sum_{j=p+1}^d\partial_j\overline{\partial_j})\beta^i\\
&=\sum_{j=1}^p i\partial_j(z_j \beta^{i-1})- \sum_{j=p+1}^di\partial_j((-z_j)\beta^{i-1})\\
&=i\sum_{j=1}^d \partial_j(z_j\beta^{i-1})\\
&=i\sum_{j=1}^d \beta^{i-1}+i(\sum_{j=1}^p (i-1)z_j\overline{z_j}+ \sum_{j=p+1}^d(i-1) z_j(-\overline{z_j}))\beta^{i-2}\\
&=id \beta^{i-1}+i(i-1)(\sum_{j=1}^p z_j\overline{z_j}- \sum_{j=p+1}^d z_j\overline{z_j})\beta^{i-2}\\
&=i(d+i-1)\beta^{i-1},
\end{align*}
whence the first term is $i(d+i-1)\beta^{i-1}\partial_\beta^i f$.
For the third term, by using that
\[
\sum_{j=1}^d x_j\frac{\partial}{\partial x_j}g(x_1,\ldots,x_d)=kg(x_1,\ldots,x_d)
\]
for any homogeneous polynomial $g$ of degree $k$ with respect to the variables $x_1$, $\ldots$, $x_d$, 
and by $\partial_\beta^i f\in\Hom(d,k-i,\ell-i)$, 
we obtain 
\begin{align*}
\sum_{j=1}^d ((\partial_j \beta^i)\overline{\partial_j}\partial_\beta^i f+(\overline{\partial_j}\beta^i)\partial_j \partial_\beta^i f)&=\sum_{j=1}^d (i\overline{z_j}\beta^{i-1}\overline{\partial_j}\partial_\beta^i f+i z_j\beta^{i-1}\partial_j \partial_\beta^i f)\\
&=i\beta^{i-1}\sum_{j=1}^d (z_j\partial_j \partial_\beta^i f+\overline{z_j}\overline{\partial_j}\partial_\beta^i f)\\
&=i(k+\ell-2i) \beta^{i-1}\partial_\beta^i f. \end{align*}
Substituting these into Eq.\ \eqref{add-partial} yields the desired result.
\end{proof}

Let $\varphi=\varphi_{k,\ell}$ be a map from $\Hom(d,k,\ell)$ to $\Hom(d,k,\ell)$ given by
\begin{align*}
\varphi(f)=\sum_{i=0}^{\min\{k,\ell\}} a_i \beta^i\partial_\beta^i f \qquad (f\in \Hom(d,k,\ell))
\end{align*}
with $a_0=1$ and $a_i+(i+1)(d+k+\ell-i-2)a_{i+1}=0$ for $i>0$, which implies 
\begin{equation}\label{eq-ai}
  a_i=\frac{(-1)^i(d+k+\ell-i-2)!}{i!(d+k+\ell-2)!}.
\end{equation}

From $a_0=1$, it follows that $\varphi(f)=f$ for $f\in\Harm_B(d,k,\ell)$.
By Lemma~\ref{lem:add-rec} and the recurrence for $a_i$, it is easy to see that $\partial_\beta(\varphi(f))=0$.
Therefore, it turns out that the map $\varphi$ is a projection onto $\Harm_B(d,k,\ell)$.

By Eq.\ \eqref{eq:add-repro} and Lemma~\ref{lem:add-kl}, it follows that 
$r_{\mathbf{x}}=\frac{1}{k!\ell!}\varphi_{k,\ell}\big(B(\mathbf{x},\cdot)^k \overline{B(\mathbf{x},\cdot)}^\ell\big)$.
Then, by Eq.\ \eqref{eq-ai} and 
\[\partial_\beta^i B(\mathbf{x},\cdot)^k \overline{B(\mathbf{x},\cdot)}^\ell=\frac{k!\ell!}{(k-i)!(\ell-i)!}\beta(\mathbf{x})^i B(\mathbf{x},\cdot)^{k-i} \overline{B(\mathbf{x},\cdot)}^{\ell-i},\] 
we have
\begin{align*}
r_\mathbf{x}&=\frac{1}{k!\ell!}\varphi_{k,\ell}(B(\mathbf{x},\cdot)^k \overline{B(\mathbf{x},\cdot)}^\ell)\\
&=\frac{1}{k!\ell!}\sum_{i=0}^{\min\{k,\ell\}}a_i\beta(\cdot)^i \partial_\beta^i \big(B(\mathbf{x},\cdot)^k \overline{B(\mathbf{x},\cdot)}^\ell\big)\\
&=\sum_{i=0}^{\min\{k,\ell\}}\frac{(-1)^i(d+k+\ell-i-2)!}{i!(k-i)!(\ell-i)!(d+k+\ell-2)!}\beta(\mathbf{x})^i\beta(\cdot)^iB(\mathbf{x},\cdot)^{k-i} \overline{B(\mathbf{x},\cdot)}^{\ell-i}\\
&=(\beta(\mathbf{x})\beta(\cdot))^{(k+\ell)/2}\sum_{i=0}^{\min\{k,\ell\}}\frac{(-1)^i(d+k+\ell-i-2)!}{i!(k-i)!(\ell-i)!(d+k+\ell-2)!}
\frac{B(\mathbf{x},\cdot)^{k-i}\overline{B(\mathbf{x},\cdot)}^{\ell-i}}{\left(\sqrt{\beta(\mathbf{x})\beta(\cdot)}\right)^{k+\ell-2i}}.
\end{align*}


Therefore, we have the following theorem.
\begin{theorem}\label{thm:add-form}
For $\mathbf{x},\mathbf{y}\in\mathbb{C}^{p,q}$ with $\beta(\mathbf{x})>0,\beta(\mathbf{y})> 0$, 
\begin{align*}
\gamma^d_{k,\ell}(\beta(\mathbf{x})\beta(\mathbf{y}))^{\frac{k+\ell}{2}}g^d_{k,\ell}\left(\frac{B(\mathbf{x},\mathbf{y})}
{\sqrt{\beta(\mathbf{x})\beta(\mathbf{y})}}\right)=
\sum_{i=1}^{\mu_{k,\ell}}\overline{f_{k,\ell,i}(\mathbf{x})}f_{k,\ell,i}(\mathbf{y})-
\sum_{i=1}^{\nu_{k,\ell}}\overline{g_{k,\ell,i}(\mathbf{x})}g_{k,\ell,i}(\mathbf{y}),
\end{align*}
where 
$g^d_{k,\ell}$ is defined by Eq.\ \eqref{eq-Jacobi},
\begin{align*}
\gamma^d_{k,\ell}&=\frac{(d+k-2)!(d+\ell-2)!}{m_{k,\ell}^d(d-2)!k!\ell!(d+k+\ell-2)!},
\end{align*}
and $m_{k,\ell}^d$ is defined by Eq.\ \eqref{eq:dim}.
\end{theorem}

\subsection{An application to Hermitian adjacency matrices}

Let $X$ be a finite set of points of $\mathbb{C}^{p,q}$, $p+q=d$, such that $B(\mathbf{x},\mathbf{x})=1$ for any $\mathbf{x}\in X$. 
Put $A(X):=\{B(\mathbf{x},\mathbf{y})\mid \mathbf{x},\mathbf{y}\in X, \mathbf{x}\neq \mathbf{y}\}$ and $n=|X|$. 
The set $X$ is an $\mathcal{S}$-\textbf{code} if there exists a two-variate polynomial 
$\phi(z)=\sum_{(k,\ell)\in \mathcal{S}}a_{k,\ell}g_{k,\ell}(z)$ (with $g_{k,\ell}=g^d_{k,\ell}$ defined by Eq.\ \eqref{eq-Jacobi}) 
in $z$ and $\bar{z}$
such that 
$a_{k,\ell}\in\mathbb{R}$  for all $(k,\ell)\in\mathcal{S}$, and $\phi(\alpha)=0$ for any $\alpha \in A(X)$, and $\phi(1)=1$. 
For an $\mathcal{S}$-code $X$, such a polynomial $\phi(x)$ is called an \textbf{annihilator polynomial} of $X$.

\begin{example}\label{ex:Scode}
Let $X$ be an $\mathcal{S}$-code of $\mathbb{C}^{p,q}$.
Define $\phi(z)=\prod_{\alpha\in A(X)}\frac{z-\alpha}{c-\alpha}$.
Then $X$ is an $\mathcal{S}$-code,  $\mathcal{S}=\{(k,0)\mid k=0,1,\ldots,|A(X)|\}$, with annihilator polynomial $\phi(z)$.
\end{example}
\begin{example}
Let $X$ be a finite set of points with $A(X)=\{\alpha,\overline{\alpha}\}$.
Define $\phi(z)=z+\overline{z}-\alpha-\overline{\alpha}$.
Then $X$ is an $\mathcal{S}$-code, where $\mathcal{S}=\{(0,0),(1,0),(0,1)\}$, with  annihilator polynomial $\phi(z)$.
With Example~\ref{ex:Scode}, it shows that the symbol $\mathcal{S}$ such that $X$ is an $\mathcal{S}$-code is not uniquely determined.
\end{example}

Define the matrices $F_{k,\ell}\in \mathbb{C}^{n\times \mu_{k,\ell}}$, 
$G_{k,\ell}\in \mathbb{C}^{n\times \nu_{k,\ell}}$, and 
$A_\alpha\in \mathbb{C}^{n\times n}$ 
with entries given by
\begin{align*}
\big(F_{k,\ell}\big)_{\mathbf{x},i}&=(f_{k,\ell,i}(\mathbf{x})),\\ 
\big(G_{k,\ell}\big)_{\mathbf{x},i}&=(g_{k,\ell,i}(\mathbf{x})),\\ 
\big(A_{\alpha}\big)_{\mathbf{x},\mathbf{y}}&=(\rho_{\alpha}(\mathbf{x},\mathbf{y})),
\end{align*}
where $\alpha \in A(X)\cup\{1\}$ and $\rho_{\alpha}(\mathbf{x},\mathbf{y})=
\begin{cases}
1 & \text{ if }B(\mathbf{x},\mathbf{y})=\alpha,\\
0 & \text{ otherwise}.
\end{cases}$.

Then Theorem~\ref{thm:add-form} reads as
\begin{align*}
F_{k,\ell}F_{k,\ell}^*-G_{k,\ell}G_{k,\ell}^*=\sum_{\alpha\in A(X)\cup\{1\}}g_{k,\ell}(\alpha)A_\alpha.
\end{align*}

Let $X$ be an $\mathcal{S}$-code with 
annihilator polynomial $\phi(z)=\sum_{(k,\ell)\in \mathcal{S}}a_{k,\ell}g_{k,\ell}(z)$.
By $\phi(1)=1$ and $\phi(\alpha)=0$ for any $\alpha\in A(X)$, we obtain that 
\begin{align*}
\sum_{(k,\ell)\in \mathcal{S}} a_{k,\ell}(F_{k,\ell}F_{k,\ell}^*-G_{k,\ell}G_{k,\ell}^*)&=
\sum_{\alpha\in A(X)\cup\{1\}}\big(\sum_{(k,\ell)\in \mathcal{S}} a_{k,\ell} g_{k,\ell}(\alpha)\big)A_{\alpha}\\
&=
\sum_{\alpha\in A(X)\cup\{1\}}\phi(\alpha)A_{\alpha}=I_n,
\end{align*}
which is also written as
\begin{align*}
H \Big(\bigoplus_{(k,\ell)\in \mathcal{S}} a_{k,\ell} I_{\mu_{k,\ell},\nu_{k,\ell}}\Big)H^*=I_n,
\end{align*}
where the matrix $H$ is defined to be $H=(F_{k,\ell}, G_{k,\ell})_{(k,\ell)\in \mathcal{S}}$.
Now we have the following lemma.

\begin{lemma}[{\cite[Lemma~2.7.1]{Bl}}]
Let $I_{s,t}=\mathrm{diag}(1^s,(-1)^t,0^{m-s-t})$ and $S$ be an $n\times m$ matrix.
If $SI_{s,t}S^*=I_n$, then $n\leq s$.
\end{lemma}

Therefore, we have the following theorem.
\begin{theorem}\label{thm:add-upperbound}
Let $X$ be an $\mathcal{S}$-code in $\mathbb{C}^{p,q}$ and 
$A(X)=\{B(\mathbf{x},\mathbf{y})\mid \mathbf{x},\mathbf{y}\in X, \mathbf{x}\neq \mathbf{y}\}$.
Assume that $B(\mathbf{x},\mathbf{x})=1$ for any $\mathbf{x}\in X$.
Let $\phi(z)=\sum_{(k,\ell)\in \mathcal{S}}a_{k,\ell}g_{k,\ell}(z)$ be an annihilator polynomial of $X$.
Then
$|X|\leq \sum_{(k,\ell)\in \mathcal{S}} \sigma_{k,\ell}$ holds, where
\begin{align*}
\sigma_{k,\ell}=\begin{cases}
\mu_{k,\ell} & \text{ if } a_{k,\ell}>0,\\
\nu_{k,\ell} & \text{ if } a_{k,\ell}<0,\\
0 & \text{ if } a_{k,\ell}=0.
\end{cases}
\end{align*}
\end{theorem}

We now apply Theorem~\ref{thm:add-upperbound} to Hermitian adjacency matrices $H$ 
in order to obtain upper bounds on the order of a digraph in terms of the inertia of $H-\lambda I$ for an eigenvalue $\lambda$ (not the rank of $H-\lambda I$, which is the same as the codimension of $\lambda$, as in Theorem \ref{theo-main}). 
Let $\Delta$ be a digraph on $n$ vertices, 
and $H:=H_{\omega}(\Delta)$, for some $\omega\in\mathbb{C}\setminus \mathbb{R}$, $|\omega|=1$, be a Hermitian adjacency matrix 
of $\Delta$. Let $\lambda_1>\lambda_2>\cdots>\lambda_s$ be the distinct eigenvalues of $H$ of respective multiplicities $m_1$, $m_2$, $\ldots$, $m_s$.

Fix $\lambda=\lambda_i$ ($2\leq i\leq s-1$), and assume that $\lambda\not\in\{0,-1\}$.
Then the matrix $I-\frac{1}{\lambda}H$ has inertia $(p,q)$, where $p=\sum_{j=i+1}^{s}m_j,q=\sum_{j=1}^{i-1}m_j$ if $\lambda>0$, or $p=\sum_{j=1}^{i-1}m_j,q=\sum_{j=i+1}^{s}m_j$ if $\lambda<0$.
Regard the matrix $I-\frac{1}{\lambda}H$ as the Gram matrix of a finite set $X$ of points in $\mathbb{C}^{p,q}$ with $A(X)=\{0,-\frac{1}{\lambda},-\frac{\omega}{\lambda},-\frac{\overline{\omega}}{\lambda}\}$ and $n=|X|$.

Put $\mathcal{S}:=\{(0,0),(1,0),(2,0),(0,1),(1,1)\}$.
Define a polynomial $F(z)$ in $z$ and $\bar{z}$ by
$F(z)=\sum_{(k,\ell)\in\mathcal{S}}a_{k,\ell}g_{k,\ell}(z)$, 
where
\begin{align*}
a_{0,0}&=\frac{(x-x^2+y^2)\lambda}{2dy^2(1+\lambda)}, \\
a_{1,0}&=\frac{1}{2d(1+\lambda)}, \\
a_{0,1}&=\frac{1}{2d(1+\lambda)}, \\
a_{1,1}&=\frac{(x-x^2+y^2)\lambda}{2d(d+1)y^2(1+\lambda)}, \\
a_{2,0}&=\frac{(1-x)\lambda}{2d(d+1)y^2(1+\lambda)},
\end{align*}
where $d=p+q$ and $\omega=x+{\bf i}y$ ($x,y\in\mathbb{R},x^2+y^2=1$).

Then $F(\alpha)=0$ for any $\alpha\in A(X)$ and $F(1)=1$,
and $F(z)\in\text{Span}\{1,z,\bar{z},z^2,z\bar{z}\}$.
Therefore  $X$ is an $\mathcal{S}$-code in $\mathbb{C}^{p,q}$ with annihilator polynomial $F$. 
Note that 
\begin{align*}
\sigma_{0,0}+\sigma_{1,1}&=\begin{cases}
\mu_{0,0}+\mu_{1,1} & \text{ if } (x>-\frac{1}{2}, \frac{\lambda}{1+\lambda}>0) \text{ or } (x<-\frac{1}{2}, \frac{\lambda}{1+\lambda}<0),\\
\nu_{0,0}+\nu_{1,1} & \text{ if } (x>-\frac{1}{2}, \frac{\lambda}{1+\lambda}<0) \text{ or } (x<-\frac{1}{2}, \frac{\lambda}{1+\lambda}>0),\\
0 & \text{ if } x=-\frac{1}{2},
\end{cases}\\
\sigma_{1,0}+\sigma_{0,1}&=\begin{cases}
\mu_{1,0}+\mu_{0,1} & \text{ if } \lambda>-1,\\
\nu_{1,0}+\nu_{0,1} & \text{ if } \lambda<-1,
\end{cases}\\
\sigma_{2,0}&=\begin{cases}\mu_{2,0} & \text{ if }\frac{\lambda}{1+\lambda}>0, \\
\nu_{2,0} & \text{ if }\frac{\lambda}{1+\lambda}<0.
\end{cases}
\end{align*}

If $\textrm{Re}(\omega)>-\frac{1}{2}$ and $0<\lambda_i$ hold, then by Theorem~\ref{thm:add-upperbound}, we obtain  
\begin{align*}
n\leq \sum_{(k,\ell)\in\mathcal{S}}\sigma_{k,\ell}=\mu_{0,0}+\mu_{1,1}+\mu_{1,0}+\mu_{0,1}+\mu_{2,0}=\frac{1}{2}(3p_i^2+3q_i^2+5p_i+q_i).
\end{align*}
The other cases follow in a similar manner; thus, we have the following result, which implies Theorem \ref{theo-main-2}. 
\begin{theorem}\label{thm:main2}
Let $H=H_\omega(\Delta)$ be a Hermitian adjacency matrix of a digraph $\Delta$ of order $n$ with $s$
distinct eigenvalues $\lambda_1>\lambda_2>\cdots >\lambda_s$ of respective multiplicities $m_1,m_2,\ldots,m_s$.
Set $p_i=\sum_{j=i+1}^{s}m_j,q_i=\sum_{j=1}^{i-1}m_j$ if $\lambda_i>0$, or $p_i=\sum_{j=1}^{i-1}m_j,q_i=\sum_{j=i+1}^{s}m_j$ if $\lambda_i<0$.
Then, for $i\in\{2,\ldots,s-1\}$ such that $\lambda_i\notin\{0,-1\}$, 
one of the following holds.
\begin{enumerate}
\item If $\textrm{Re}(\omega)<-\frac{1}{2}$ and $\lambda_i<-1$, then
\begin{align*}
n\leq \nu_{0,0}+\nu_{1,0}+\nu_{0,1}+\nu_{1,1}+\mu_{2,0}=\frac{1}{2} \left(p_i^2+4 p_i q_i+q_i^2+p_i +5q_i\right).
\end{align*}
\item If $\textrm{Re}(\omega)<-\frac{1}{2}$ and $-1<\lambda_i<0$, then
\begin{align*}
n\leq \mu_{0,0}+\nu_{1,0}+\nu_{0,1}+\mu_{1,1}+\nu_{2,0}=p_i^2+q_i^2+p_iq_i+2p_i.
\end{align*}
\item If $\textrm{Re}(\omega)<-\frac{1}{2}$ and $0<\lambda_i$, then
\begin{align*}
n\leq \nu_{0,0}+\mu_{1,0}+\mu_{0,1}+\nu_{1,1}+\mu_{2,0}=\frac{1}{2} \left(p_i^2+4 p_i q_i+q_i^2+5 p_i+q_i\right).
\end{align*}
\item If $\textrm{Re}(\omega)=-\frac{1}{2}$ and $\lambda_i<-1$, then
\begin{align*}
n\leq \nu_{1,0}+\nu_{0,1}+\mu_{2,0}=\frac{1}{2}(p_i^2+q_i^2+p_i+5q_i).
\end{align*}
\item If $\textrm{Re}(\omega)=-\frac{1}{2}$ and $0<\lambda_i<-1$, then
\begin{align*}
n\leq \mu_{1,0}+\mu_{0,1}+\nu_{2,0}= p_i q_i+2 p_i.
\end{align*}
\item If $\textrm{Re}(\omega)=-\frac{1}{2}$ and $0<\lambda_i$, then
\begin{align*}
n\leq \mu_{1,0}+\mu_{0,1}+\mu_{2,0}=\frac{1}{2} \left(p_i^2+q_i^2+5p_i +q_i\right).
\end{align*}
\item If $\textrm{Re}(\omega)>-\frac{1}{2}$ and $\lambda_i<-1$, then
\begin{align*}
n\leq \mu_{0,0}+\nu_{1,0}+\nu_{0,1}+\mu_{1,1}+\mu_{2,0}=\frac{1}{2} \left(3p_i^2+3q_i^2+p_i +5q_i\right).
\end{align*}
\item If $\textrm{Re}(\omega)>-\frac{1}{2}$ and $-1<\lambda_i<0$, then
\begin{align*}
n\leq \nu_{0,0}+\mu_{1,0}+\mu_{0,1}+\nu_{1,1}+\nu_{2,0}=3p_i q_i+2p_i.
\end{align*}
\item If $\textrm{Re}(\omega)>-\frac{1}{2}$ and $0<\lambda_i$, then
\begin{align*}
n\leq \mu_{0,0}+\mu_{1,0}+\mu_{0,1}+\mu_{1,1}+\mu_{2,0}=\frac{1}{2} \left(3p_i^2+3q_i^2+5p_i +q_i\right).
\end{align*}
\end{enumerate}
\end{theorem}

\section{Concluding remarks}\label{sect:final}

By Theorem \ref{theo-BR-2}, a graph attaining equality
in Eq.\ \eqref{eq-DGS} is extremal strongly regular (see also \cite{Rsign}
for its signed analogue). It would be interesting to see whether
graphs (if any exist) attaining equality in the bounds in Theorem \ref{theo-main}
possess any remarkable combinatorial properties.
However, we expect this to be a difficult question  because
there are several different bounds (unlike the case of $(0,1)$-adjacency matrices).
In particular, it is not clear why the case $\omega=\frac{-1\pm{\bf i}\sqrt{3}}{2}$
is somewhat special so that the bound in Theorem \ref{th-case2} is better than that
in Theorem \ref{th-case1} (although we could not improve the former one by $-1$
as in Theorem \ref{th-case1-improv}).

\begin{problem}
What kind of structures (e.g., association schemes) are digraphs achieving the bounds in Theorems $\ref{theo-main}$ or  $\ref{theo-main-2}$ related to?
\end{problem}

Note that \cite[Theorem~2.3]{BR} extends the bound from Theorem \ref{theo-BR-2} to
all eigenvalues distinct from $0$ and $-1$.
Its proof relies on the existence of the Perron-Frobenius eigenvector
with all positive coordinates. Since Hermitian matrices do not have such an eigenvector in general,
it is not clear how to generalize the star complement technique in Section \ref{ssect:starbound}
to main eigenvalues.
Accordingly, in this situation Theorem \ref{theo-main} does not give a bound in $d_i$
when $i\notin \{1,s\}$.
However, one can use the following variant of \cite[Lemma~10]{Bukh},
which may give a better bound (than that in Eq.\ \eqref{eq-Bukh}) when $L$ contains conjugate elements.
To apply Lemma \ref{lemma-Bukh}, one should follow the proof of \cite[Proposition~1]{Bukh}.

\begin{lemma}\label{lemma-Bukh}
Suppose that $M=(m_{ij})$ is a matrix over $\mathbb{C}$.
Let $f=f(z,\bar{z})=\sum_{(k,\ell)\in \mathcal{S}}a_{k,\ell}z^k \bar{z}^{\ell}$,
where $\mathcal{S}$ is a set of pairs of nonnegative integers,
be a two-variate polynomial in $z$ and $\bar{z}$.
Define $f[M]$ to be the matrix obtained from $M$ by applying $f$
to each of its entries, i.e., $(f[M])_{ij}=f(m_{ij},\overline{m_{ij}})$.
Then
\begin{equation*}
\rank(f[M])\leq \sum_{(k,\ell)\in\mathcal{S}}\binom{\rank(M)+k-1}{k}\binom{\rank(M)+\ell-1}{\ell}.
\end{equation*}
\end{lemma}
\begin{proof}
Let $r=\rank(M)$ and $v_1,\ldots,v_n$ be the columns of $M$, and assume that
$v_1,\ldots,v_r$ span the column space of $M$.
For nonnegative integers $a,b$, define
\begin{align*}
V_{a,b}=\text{span}\{v_1^{e_1}\cdots v_r^{e_r}\overline{v_1}^{f_1}\cdots \overline{v_r}^{f_r} \mid e_i,f_j\in\mathbb{N}, 
\sum_{i=1}^r e_i=a, \sum_{i=1}^r f_i=b\},
\end{align*}
where $v^k$ stands for the coordinate-wise power of $v$ and $v_jv_j$ means 
the coordinate-wise product of $v_i$ and $v_j$. 

Observe that $\dim V_{a,b}=\binom{r+a-1}{a}\binom{r+b-1}{b}$.
Let $W$ be the span of all $V_{k,\ell}$ with $(k,\ell)\in\mathcal{S}$.
Since $v_i=\sum_{j=1}^r \alpha_{ij}v_j$ for some scalars $\alpha_{ij}$, it follows that
the coordinate-wise product of $v_i^a$ and $\overline{v_i}^b$, which is given by
\begin{equation*}
v_i^a \overline{v_i}^b=\sum_{j_1,\ldots,j_a,j'_1,\ldots,j'_b}\prod_{x=1}^{a}\prod_{y=1}^b \alpha_{ij_x}\overline{\alpha_{ij'_y}}v_{j_x}\overline{v_{j'_y}},
\end{equation*}
belongs to $V_{a,b}$. Hence each column of $f[M]$ lies in $W$.
Thus, 
\begin{align*}
\rank(f[M])\leq \dim W=\sum_{(k,\ell)\in\mathcal{S}}\dim V_{k,\ell}=\sum_{(k,\ell)\in\mathcal{S}}\binom{r+k-1}{k}\binom{r+\ell-1}{\ell},
\end{align*}
and the lemma follows.
\end{proof}

Now let us discuss a generalization of the Neumaier's bound. More precisely, \cite{Neu} says that the order $n$ of a strongly regular graph with eigenvalues $k,\theta_1>0,\theta_2<0$ of respective multiplicities $1,m_1,m_2$ satisfies:
\begin{equation}\label{eq-abs-bound-Neu2}
    n\leq \frac{1}{2}m_i(m_i+1)
\end{equation}
unless equality in the corresponding Krein condition holds:
\begin{equation}\label{eq-Krein}
(\theta_i + 1)(k + \theta_i + 2\theta_i\theta_j)=(k + \theta_i)(\theta_j + 1)^2, \quad \{i,j\}=\{1,2\}.
\end{equation}

We have the following result.

\begin{theorem}\label{theo-Neu}
Let $A$ denote a $(0,1)$-adjacency matrix of a connected $k$-regular (simple) graph on $n$ vertices.
Let $\lambda$ be an eigenvalue of $A$ 
of multiplicity $n-d$, $d\geq 2$, 
such that $\lambda\notin\{0,-1,k\}$ and 
$A$ has no eigenvalue $\frac{\lambda(2\lambda-2k- \lambda n)}{2\lambda - 2k + n}$. 
Then
\begin{equation}\label{eq-genNeu}
 n\leq \frac{1}{2}d(d-1).
\end{equation}
\end{theorem}
\begin{proof}
  As $\Gamma$ is regular and connected, the matrix $A$ has the only main eigenvalue, which is $k$, of multiplicity $1$
  with eigenvector $\mathbf{1}$. Then the matrix $G:=A+\frac{\lambda-k}{n}J$ has eigenvalue $\lambda$ of multiplicity $n-d+1$.
  Consider a star complement for $\lambda$ in $G$, which has size $d-1$, and, following the notation of Section \ref{ssect:star}
  and the proof of Theorem \ref{theo-BR-2} \cite{BR}, define the functions $F_1,\ldots,F_n$ by:
  \begin{equation*}
    F_u({\bf x})=\langle {\bf s}_u,{\bf x}\rangle^2\quad(u\in V,~{\bf x}\in \mathbb{R}^{d-1}),
  \end{equation*}
  where by Eqs. \eqref{eq-inprod} and  \eqref{eq-H}, for all vertices $u,v$ of $\Gamma$, we have:
  \begin{equation*}
    \langle {\bf s}_u,{\bf s}_v\rangle = \left\{\begin{matrix}
                                               \lambda-\frac{\lambda-k}{n} & \text{if~}u=v, \\
                                               -1-\frac{\lambda-k}{n} & \text{if~}u\sim v, \\
                                               -\frac{\lambda-k}{n}  & \text{~otherwise.}
                                             \end{matrix}\right.
  \end{equation*}

  Suppose that $F_1,F_2,\ldots,F_n$ are linearly dependent. Then, as in the proof of Theorem \ref{th-case1},
  we obtain that
  \begin{equation*}
    \big(c_2J+(c_0-c_2)I+(c_1-c_2)A\big)\mathbf{b}=\mathbf{0}
  \end{equation*}
  holds for some nonzero vector $\mathbf{b}$, where
  $c_0=\big(\lambda-\frac{\lambda-k}{n}\big)^2$, $c_1=\big(1+\frac{\lambda-k}{n}\big)^2$, and
  $c_2=\big(\frac{\lambda-k}{n}\big)^2$.
  It follows that $\mathbf{b}$ is an eigenvector of $c_2J+(c_1-c_2)A$ with eigenvalue $c_2-c_0$. Since $\Gamma$ is regular, the matrices $J$ and $A$ commute and hence share the eigenspaces. 
  Thus, $\mathbf{b}$ is an eigenvector of $A$, and either $\mathbf{b}$ is a scalar multiple of $\mathbf{1}$ or $\mathbf{b}\perp \mathbf{1}$. 

  In the former case, we obtain that $c_2-c_0=c_2n+(c_1-c_2)k$, which simplifies to 
  $\lambda^2(n-1) + 2\lambda k + k(n-k)=0$. 
  Then $\lambda=\frac{-k\pm\sqrt{k^2-(n-1)k(n-k)}}{n-1}$; now $k^2-(n-1)k(n-k)\geq 0$ implies $k\geq (n-1)(n-k)$, which is impossible unless $n=k+1$. If $n=k+1$, then $\Gamma$ is a complete graph, whose only eigenvalues are $k$ of multiplicity $1$ and $-1$ of multiplicity $k$, a contradiction.

  In the latter case, $\frac{c_2-c_0}{c_1-c_2}=\frac{\lambda(2\lambda-2k- \lambda n)}{2\lambda - 2k + n}$ is an eigenvalue of $A$,
  contrary to the assumption. This shows that $F_1,F_2,\ldots,F_n$ are linearly independent,
  and the result follows by the same argument as in the proof of Theorem \ref{th-case1}.
\end{proof}

Suppose now that $\Gamma$ is a strongly regular graph with the spectrum as above.
Then, by Theorem \ref{theo-Neu}, unless 
\begin{equation}\label{eq-gKrein}
\theta_i=\frac{\theta_j(2\theta_j-2k- \theta_j n)}{2\theta_j - 2k + n}
\end{equation} 
holds, we have that  
$n\leq \frac{1}{2}m_i(m_i+1)$, since 
$d=n-m_j=m_i+1$ if $\lambda=\theta_j$; i.e., Eq.\ \eqref{eq-abs-bound-Neu2} holds. 
Further, by using the standard relations between the parameters and eigenvalues of strongly regular graphs (see, e.g., \cite{CGS}), 
one can verify that Eq.\ \eqref{eq-Krein} 
implies Eq.\ \eqref{eq-gKrein} and vice versa. 

\begin{problem}
Are there non-strongly regular graphs achieving equality in Eq.\  \eqref{eq-genNeu}?
\end{problem}

\begin{problem}
What are the parameters of strongly regular graphs attaining equality in Eq.  \eqref{eq-abs-bound-Neu2}?
(All parameters of strongly regular graphs that attain equality in the bound from Theorem $\ref{theo-BR-2}$ are expressed in the eigenvalue of largest multiplicity.)
\end{problem}

\begin{problem}
Eq. \eqref{eq-abs-bound-Neu2} is a special case of the inequalities for multiplicities of association schemes shown in \cite{Neu}. Can the idea of the proof of Theorem $\ref{theo-Neu}$ be used to prove these general inequalities (and to investigate 
the case of equality; cf. Problem in \cite[Section~2.3]{BCN})?
\end{problem}

\bibliographystyle{abbrv}
\bibliography{references}

\Acknowledgements

\bibliographystyle{abbrv}
\bibliography{references}

\end{document}